\documentclass{article}
\PassOptionsToPackage{numbers, compress}{natbib}
\usepackage[final]{nips_2017}

\usepackage{graphicx} 
\usepackage{subfigure}
\usepackage{epsfig,epsf,fancybox,epstopdf}

\usepackage{algorithm}
\usepackage{algorithmic}

\usepackage{amsmath}
\usepackage{mathrsfs}
\usepackage{amssymb}
\usepackage{amsthm}

\usepackage{color}

\numberwithin{equation}{section}

\newtheorem{theorem}{Theorem}[section]

\newtheorem{lemma}[theorem]{Lemma}

\newtheorem{remark}[theorem]{Remark}

\newcommand{\argmin}{\mathop{\rm argmin}}
\newcommand{\br}{\mathbb{R}}

\newcommand{\Prox}{{\rm Prox}}
\newcommand{\half}{\frac{1}{2}}
\newcommand{\Line}{\mathop{\rm Line}}
\newcommand{\Tr}{\mathrm{Tr}}
\newcommand{\st}{\mathrm{s.t. }}

\newcommand{\be}{\begin{equation}}
\newcommand{\ee}{\end{equation}}
\newcommand{\ba}{\begin{array}}
\newcommand{\ea}{\end{array}}

\newcommand{\normone}[1]{\| #1 \|_1}
\newcommand{\normtwo}[1]{\| #1 \|}
\newcommand{\abs}[1]{| #1 |}

\newcommand{\inp}[2]{\langle #1, #2 \rangle}

\usepackage[utf8]{inputenc} 
\usepackage[T1]{fontenc}    
\usepackage{hyperref}       
\usepackage{url}
\usepackage{booktabs}       
\usepackage{amsfonts}       
\usepackage{nicefrac}       
\usepackage{microtype}

\usepackage{authblk}
\usepackage{bigstrut}

\title{Geometric Descent Method for Convex Composite Minimization }
\author[1]{Shixiang Chen}
\author[1]{Shiqian Ma}
\author[2]{Wei Liu}
\affil[1]{Department of SEEM, The Chinese University of Hong Kong, Hong Kong}
\affil[2]{Tencent AI Lab, Shenzhen, China}

\begin{document}
\maketitle

\begin{abstract}
In this paper, we extend the geometric descent method recently proposed by Bubeck, Lee and Singh \cite{bubeck2015geometric} to tackle nonsmooth and strongly convex composite problems. We prove that our proposed algorithm, dubbed geometric proximal gradient method (GeoPG), converges with a linear rate $(1-1/\sqrt{\kappa})$ and thus achieves the optimal rate among first-order methods, where $\kappa$ is the condition number of the problem. Numerical results on linear regression and logistic regression with elastic net regularization show that GeoPG compares favorably with Nesterov's accelerated proximal gradient method, especially when the problem is ill-conditioned.
\end{abstract}

\section{Introduction}
Recently, Bubeck, Lee and Singh proposed a geometric descent method (GeoD) for minimizing a smooth and strongly convex function \cite{bubeck2015geometric}. They showed that GeoD achieves the same optimal rate as Nesterov's accelerated gradient method (AGM) \cite{Nesterov-1983,NesterovConvexBook2004}. In this paper, we provide an extension of GeoD that minimizes a nonsmooth function in the composite form:
\be\label{composite}
\min_{x\in\br^n} \ F(x) := f(x)+h(x),
\ee
where $f$ is $\alpha$-strongly convex and $\beta$-smooth (\textit{i.e.}, $\nabla f$ is Lipschitz continuous with Lipschitz constant $\beta$), and $h$ is a closed nonsmooth convex function with simple proximal mapping. Commonly seen examples of $h$ include $\ell_1$ norm, $\ell_2$ norm, nuclear norm, and so on.

If $h$ vanishes, then the objective function of \eqref{composite} becomes smooth and strongly convex. In this case, it is known that AGM converges with a linear rate $(1-1/\sqrt{\kappa})$, which is optimal among all first-order methods, where $\kappa=\beta/\alpha$ is the condition number of the problem. However, AGM lacks a clear geometric intuition, making it difficult to interpret. Recently, there has been much work on attempting to explain AGM or designing new algorithms with the same optimal rate (see,
\cite{Su-Boyd-Candes-NIPS-2014,Attouch-2016-MP,bubeck2015geometric,Lessard-Recht-Packard-siopt-2016,Wibisono-PNAS-2016}). In particular, the GeoD method proposed in \cite{bubeck2015geometric} has a clear geometric intuition that is in the flavor of the ellipsoid method \cite{bland-goldfarb-todd-ellipsoid-survey}. The follow-up work \cite{bubeck2016black,drusvyatskiy2016optimal} attempted to improve the performance of GeoD by exploiting the gradient information from the past with a ``limited-memory'' idea. Moreover, Drusvyatskiy, Fazel and Roy \cite{drusvyatskiy2016optimal} showed how to extend the suboptimal version of GeoD (with the convergence rate $(1-1/\kappa)$) to solve the composite problem \eqref{composite}. However, it was not clear how to extend the optimal version of GeoD to address  \eqref{composite}, and the authors posed this as an open question. In this paper, we settle this question by proposing a geometric proximal gradient (GeoPG) algorithm which can solve the composite problem \eqref{composite}. We further show how to incorporate various techniques to improve the performance of the proposed algorithm.

{\bf Notation.} We use $B(c,r^2)=\left\lbrace  x|\normtwo{x-c}^2\leq r^2\right\rbrace $ to denote the ball with center $c$ and radius $r$. We use $\Line(x,y)$ to denote the line that connects $x$ and $y$, \textit{i.e.}, $\{x+ s(y-x), s\in\br\}$.
{For fixed $t\in(0,1/\beta]$, we denote $x^+ := \Prox_{th}(x-t\nabla f(x))$, where the proximal mapping $\Prox_h(\cdot)$ is defined as $\Prox_h(x) := \argmin_z \ h(z) + \half\|z-x\|^2$. The proximal gradient of $F$ is defined as $G_t(x):=(x-x^+)/t$. It should be noted that $x^+ = x - tG_t(x)$. We also denote $x^{++} := x - G_t(x)/\alpha$. Note that both $x^+$ and $x^{++}$ are related to $t$, and we omit $t$ whenever there is no ambiguity.}


The rest of this paper is organized as follows. In Section \ref{sec:geoM}, we briefly review the GeoD method for solving smooth and strongly convex problems. In Section \ref{sec:GeoPG}, we provide our GeoPG algorithm for solving nonsmooth problem \eqref{composite} and analyze its convergence rate. We address two practical issues of the proposed method in Section \ref{sec:practical}, and incorporate two techniques: backtracking and limited memory, to cope with these issues. In Section \ref{sec:num}, we report some numerical results of comparing GeoPG with Nesterov's accelerated proximal gradient method in solving linear regression and logistic regression problems with elastic net regularization. Finally, we conclude the paper in Section \ref{sec:conclusion}.

\section{Geometric Descent Method for Smooth Problems}\label{sec:geoM}

{The GeoD method \cite{bubeck2015geometric} solves \eqref{composite} when $h\equiv 0$, in which the problem reduces to a smooth and strongly convex problem $\min \ f(x)$. We denote its optimal solution and optimal value as $x^*$ and $f^*$, respectively.
Throughout this section, we fix $t =1/\beta$, which together with $h\equiv 0$ implies that $x^+ = x-\nabla f(x)/\beta$ and $x^{++} = x-\nabla f(x)/\alpha$.}
We first briefly describe the basic idea of the suboptimal GeoD. Since $f$ is $\alpha$-strongly convex, the following inequality holds
\begin{equation}\label{sc-lower-bound}
f(x) + \inp{\nabla f(x)}{y-x} +\frac{\alpha}{2}\normtwo{y-x}^2 \leq f(y), \ \forall x,y\in\br^n.
\end{equation}
By letting $y=x^*$ in \eqref{sc-lower-bound}, one obtains that
\be\label{sc-lower-bound-ball}
x^*\in B\big(x^{++}, \normtwo{\nabla f(x)}^2/\alpha^2-2(f(x)-f^*)/\alpha\big), \forall x\in\br^n.
\ee
Note that the $\beta$-smoothness of $f$ implies
\be\label{beta-smooth-inequality}
f(x^+) \leq f(x) - \|\nabla f(x)\|^2/(2\beta), \forall x\in\br^n.
\ee
Combining \eqref{sc-lower-bound-ball} and \eqref{beta-smooth-inequality} yields
$
x^*\in B\big(x^{++}, (1-1/\kappa)\normtwo{\nabla f(x)}^2/\alpha^2-2(f(x^+)-f^*)/\alpha\big).
$
As a result, suppose that initially we have a ball $B(x_0,R_0^2)$ that contains $x^*$, then it follows that
\be\label{geoM-basic-two-intersection}
x^*\in B\big(x_0,R_0^2\big) \cap
  B\big(x_0^{++}, (1-1/\kappa)\normtwo{\nabla f(x_0)}^2/\alpha^2-2(f(x_0^+)-f^*)/\alpha\big).
\ee
Some simple algebraic calculations show that the squared radius of the minimum enclosing ball of the right hand side of \eqref{geoM-basic-two-intersection} is no larger than $R_0^2(1-1/\kappa)$, \textit{i.e.}, there exists some $x_1\in\br^n$ such that $x^*\in B\big(x_1,R_0^2(1-1/\kappa)\big)$. Therefore, the squared radius of the initial ball shrinks by a factor $(1-1/\kappa)$. Repeating this process yields a linear convergent sequence $\{x_k\}$ with the convergence rate $(1-1/\kappa)$:
$
\|x_k-x^*\|^2 \leq (1-1/\kappa)^k R_0^2.
$

The optimal GeoD (with the linear convergence rate $(1-1/\sqrt{\kappa})$) maintains two balls containing $x^*$ in each iteration, whose centers are $c_k$ and $x_{k+1}^{++}$, respectively. More specifically, suppose that in the $k$-th iteration we have $c_k$ and $x_k$, then $c_{k+1}$ and $x_{k+1}$ are obtained as follows. First, $x_{k+1}$ is the minimizer of $f$ on $\Line(c_k,x_k^+)$. Second, $c_{k+1}$ (resp. $R_{k+1}^2$) is the center (resp. squared radius) of the ball (given by Lemma \ref{lem:enclosing-ball}) that contains
\[
\begin{aligned}
B(c_k, R_k^2 - \|\nabla f(x_{k+1})\|^2/(\alpha^2\kappa)) \cap B(x_{k+1}^{++}, (1-1/\kappa)\|\nabla f(x_{k+1})\|^2/\alpha^2).
\end{aligned}
\]
Calculating $c_{k+1}$ and $R_{k+1}$ is easy and we refer to Algorithm 1 of \cite{bubeck2015geometric} for details. By applying Lemma \ref{lem:enclosing-ball} with $x_A = c_k$, $r_A = R_k$, $r_B = \|\nabla f(x_{k+1})\|/\alpha$, $\epsilon = 1/\kappa$ and $\delta=\frac{2}{\alpha}(f(x_k^+)-f(x^*))$, we obtain $R_{k+1}^2 = (1-1/\sqrt{\kappa})R_k^2$, which further implies
$\|x^*-c_k\|^2 \leq (1-1/\sqrt{\kappa})^k R_0^2,$
\textit{i.e.}, the optimal GeoD converges with the linear rate $(1-1/\sqrt{\kappa})$.

\begin{lemma}[see \cite{bubeck2015geometric,drusvyatskiy2016optimal}]\label{lem:enclosing-ball}
Fix centers $x_A$, $x_B\in \mathbb{R}^n$ and squared radii $r_A^2, r_B^2>0$. Also fix $\epsilon\in(0,1)$ and suppose $\normtwo{x_A-x_B}^2\geq r_B^2$. There exists a new center $c\in \mathbb{R}^n$ such that for any $\delta >0$, we have
\[
\begin{aligned}
  B(x_A ,r_A^2-\epsilon r_B^2-\delta)\cap B\big(x_B,r_B^2(1-\epsilon)-\delta\big)  \subset B\big(c,(1-\sqrt{\epsilon})r_A^2-\delta\big).
\end{aligned}
\]
\end{lemma}

\section{Geometric Descent Method for Convex Nonsmooth Composite Problems}\label{sec:GeoPG}
Drusvyatskiy, Fazel and Roy \cite{drusvyatskiy2016optimal} extended the suboptimal GeoD to solve the composite problem \eqref{composite}. However, it was not clear how to extend the optimal GeoD to solve problem \eqref{composite}. We resolve this problem in this section.
%

The following lemma is useful to our analysis. Its proof is in the appendix.

\begin{lemma}\label{lem1}
Given point $x\in\mathbb{R}^n$ and step size $t \in (0, 1/\beta]$, denote $x^+ = x - tG_t(x)$. The following inequality holds for any $y\in \br^n$: \begin{equation}\label{maineq}
F(y)\geq F(x^+)+\inp{G_t(x)}{y-x}+\frac{t}{2}\normtwo{G_t(x)}^2+\frac{\alpha}{2}\normtwo{y-x}^2.
\end{equation}
\end{lemma}

\subsection{GeoPG Algorithm}\label{sec:GeoPG-Algorithm}
In this subsection, we describe our proposed geometric proximal gradient method (GeoPG) for solving \eqref{composite}.
Throughout Sections \ref{sec:GeoPG-Algorithm} and \ref{sec:GeoPG-convergence}, $t\in(0,1/\beta]$ is a fixed scalar.
The key observation for designing GeoPG is that in the $k$-th iteration one has to find $x_k$ that lies on $\Line(x_{k-1}^+,c_{k-1})$ such that the following two inequalities hold:
\begin{equation}\label{equ23}
\begin{aligned}
F(x_k^+)\leq F(x_{k-1}^+)-\frac{t}{2}\normtwo{G_{t}(x_k)}^2, \mbox{ and } \normtwo{x_k^{++}-c_{k-1}}^2\geq \frac{1}{\alpha^2}\normtwo{G_{t}(x_k)}^2.
\end{aligned}
\end{equation}
{Intuitively, the first inequality in \eqref{equ23} requires that there is a function value reduction on $x_k^+$ from $x_{k-1}^+$, and the second inequality requires that the centers of the two balls are far away from each other so that Lemma \ref{lem:enclosing-ball} can be applied.}

The following lemma gives a sufficient condition for \eqref{equ23}. Its proof is
in the appendix.
\begin{lemma}\label{lem:equ23-reduced}
\eqref{equ23} holds if $x_k$ satisfies
\be\label{equ23-reduced}
\langle x_k^+ - x_k, x_{k-1}^+-x_k \rangle \leq 0, \mbox{ and } \langle x_k^+ - x_k, x_k - c_{k-1}\rangle \geq 0.
\ee
\end{lemma}
Therefore, we only need to find $x_k$ such that \eqref{equ23-reduced} holds. To do so, we define the following functions for given $x$, $c$ ($x\neq c$) and $t\in(0,\beta]$:
\[
\begin{aligned}
\phi_{t,x,c}(z) = \langle z^+ - z, x-c \rangle, \forall z\in\br^n \mbox{ and } \bar\phi_{t,x,c}(s) = \phi_{t,x,c}\big(x + s(c-x)\big), \forall s\in\br.
\end{aligned}
\]
The functions $\phi_{t,x,c}(z)$ and $\bar\phi_{t,x,c}(s)$ have the following properties. Its proof can be found
in the appendix.
\begin{lemma}\label{phi-t-property}
(i) $\phi_{t,x,c}(z)$ is Lipschitz continuous. (ii) $\bar\phi_{t,x,c}(s)$ strictly monotonically increases.
\end{lemma}
%

We are now ready to describe how to find $x_k$ such that \eqref{equ23-reduced} holds. This is summarized in Lemma \ref{lemcase}.
\begin{lemma}\label{lemcase}
The following two ways find $x_k$ satisfying \eqref{equ23-reduced}. \noindent
\begin{itemize}
\item[(i)] If $\bar\phi_{t,x_{k-1}^+,c_{k-1}}(1) \leq 0$, then \eqref{equ23-reduced} holds by setting $x_k:=c_{k-1}$; if $\bar\phi_{t,x_{k-1}^+,c_{k-1}}(0) \geq 0$, then \eqref{equ23-reduced} holds by setting $x_k:=x_{k-1}^+$; if $\bar\phi_{t,x_{k-1}^+,c_{k-1}}(1)> 0$ and $\bar\phi_{t,x_{k-1}^+,c_{k-1}}(0)<0$, then there exists $s\in [0,1]$ such that $\bar\phi_{t,x_{k-1}^+,c_{k-1}}(s) = 0$. As a result, \eqref{equ23-reduced} holds by setting $x_k:=x_{k-1}^+ + s(c_{k-1}-x_{k-1}^+)$.
\item[(ii)] If $\bar\phi_{t,x_{k-1}^+,c_{k-1}}(0) \geq 0$, then \eqref{equ23-reduced} holds by setting $x_k:=x_{k-1}^+$; if $\bar\phi_{t,x_{k-1}^+,c_{k-1}}(0)< 0$, then there exists $s\geq 0$ such that $\bar\phi_{t,x_{k-1}^+,c_{k-1}}(s) = 0$. As a result, \eqref{equ23-reduced} holds by setting $x_k:=x_{k-1}^+ + s(c_{k-1}-x_{k-1}^+)$.
\end{itemize}
\end{lemma}
\begin{proof}
Case (i) directly follows from the Mean-Value Theorem. Case (ii) follows from the monotonicity and continuity of $\bar\phi_{t,x_{k-1}^+,c_{k-1}}$ from Lemma \ref{phi-t-property}.
\end{proof}

It is indeed very easy to find $x_k$ satisfying the two cases in Lemma \ref{lemcase}. Specifically, for case (i) of Lemma \ref{lemcase}, we can use the bisection method to find the zero of $\bar\phi_{t,x_{k-1}^+,c_{k-1}}$ in the closed interval $[0,1]$. In practice, we found that the Brent-Dekker method \cite{brent2013algorithms,Dekker-1969} performs much better than the bisection method, so we use the Brent-Dekker method in our numerical experiments. For case (ii) of Lemma \ref{lemcase}, we can use the semi-smooth Newton method to find the zero of $\bar\phi_{t,x_{k-1}^+,c_{k-1}}$ in the interval $[0,+\infty)$. 
In our numerical experiments, we implemented the global semi-smooth Newton method \cite{Gerdts2017,hans2015global} and obtained very encouraging results.
These two procedures are described in Algorithms \ref{Alg:Brent} and \ref{Alg:semi-smooth-newton}, respectively. Based on the discussions above, we know that $x_k$ generated by these two algorithms satisfies \eqref{equ23-reduced} and hence \eqref{equ23}.

\begin{algorithm}[h]
\caption{The first procedure for finding $x_k$ from given $x_{k-1}^+$ and $c_{k-1}$.}\label{Alg:Brent}
\begin{algorithmic}[1]
    \IF{$\inp{(x_{k-1}^+)^+-x_{k-1}^+}{x_{k-1}^+-c_{k-1}}\geq 0$}
    \STATE{set $x_k:=x_{k-1}^+$;}
    \ELSIF{$\inp{c_{k-1}^+-c_{k-1}}{x_{k-1}^+-c_{k-1}}\leq 0$}
    \STATE{set $x_k:=c_{k-1}$;}
    \ELSE
    \STATE{use the Brent-Dekker method to find $s\in [0,1]$ such that $\bar\phi_{t,x_{k-1}^+,c_{k-1}}(s)=0$ and set $x_k:= x_{k-1}^+ + s(c_{k-1}-x_{k-1}^+)$;}
     \ENDIF
  \end{algorithmic}
\end{algorithm}

\begin{algorithm}[h]
\caption{The second procedure for finding $x_k$ from given $x_{k-1}^+$ and $c_{k-1}$.}\label{Alg:semi-smooth-newton}
\begin{algorithmic}[1]
    \IF{$\inp{(x_{k-1}^+)^+-x_{k-1}^+}{x_{k-1}^+-c_{k-1}}\geq 0$}
    \STATE{set $x_k:=x_{k-1}^+$;}
    \ELSE
    \STATE{use the global semi-smooth Newton method \cite{Gerdts2017,hans2015global} to find the root $s\in [0,+\infty)$ of $\bar\phi_{t,x_{k-1}^+,c_{k-1}}(s)$, and set $x_k:= x_{k-1}^+ +s(c_{k-1}-x_{k-1}^+)$;}
    \ENDIF
\end{algorithmic}
\end{algorithm}

We are now ready to present our GeoPG algorithm for solving \eqref{composite} as in Algorithm \ref{Alg:GeoPG}. 

\begin{algorithm}[h]
\caption{GeoPG: geometric proximal gradient descent for convex composite minimization.}\label{Alg:GeoPG}
\begin{algorithmic}[1]
    \REQUIRE{Parameters $\alpha$, $\beta$, initial point $x_0$ and step size $t\in(0,1/\beta]$.}
    \STATE{Set $c_0=x_0^{++}$, $R_0^2=\normtwo{G_{t}(x_0)}^2(1-\alpha t)/\alpha^2$;}
    \FOR{$k=1,2,\ldots$}
    \STATE{Use Algorithm \ref{Alg:Brent} or \ref{Alg:semi-smooth-newton} to find $x_k$;}
    \STATE{Set $x_A := x_k ^{++}=x_k-G_{t}(x_k)/\alpha$, and $R_A^2=\normtwo{G_{t}(x_k)}^2(1-\alpha t)/\alpha^2$;}
    \STATE{Set $x_B := c_{k-1}$, and $R_B^2=R_{k-1}^2- 2(F(x_{k-1}^+)-F(x_k^+))/\alpha$;}
    \STATE{Compute $B(c_k,R_k^2)$: the minimum enclosing ball of $B(x_A,R_A^2)\cap B(x_B,R_B^2)$, which can be done using Algorithm 1 in \cite{bubeck2015geometric};}
    \ENDFOR
\end{algorithmic}
\end{algorithm}

\subsection{Convergence Analysis of GeoPG}\label{sec:GeoPG-convergence}
We are now ready to present our main convergence result for GeoPG. 
\begin{theorem}\label{the1}
Given initial point $x_0$ and step size $t\in(0,1/\beta]$, we set $R_0^2=\frac{\normtwo{G_{t}(x_0)}^2}{\alpha^2}(1-\alpha t)$. Suppose that sequence $\{(x_k,c_k,R_k)\}$ is generated by Algorithm \ref{Alg:GeoPG}, and that $x^*$ is the optimal solution of \eqref{composite} and $F^*$ is the optimal objective value.
For any $k\geq 0$, one has $x^*\in B(c_k,R_k^2)$ and $R_{k+1}^2\leq (1-\sqrt{\alpha t})R_k^2$, and thus
\be\label{the1-inequality}
\begin{aligned}
\normtwo{x^*-c_k}^2\leq (1-\sqrt{\alpha t})^k R_0^2, \mbox{ and }
 F(x_{k+1}^+)-F^* \leq \frac{\alpha}{2}(1-\sqrt{\alpha t})^k R_0^2.
 \end{aligned}
\ee
Note that when $t=1/\beta$, \eqref{the1-inequality} implies the linear convergence rate $(1-1/\sqrt{\kappa})$.
\end{theorem}
\begin{proof}
We prove a stronger result by induction that for every $k\geq 0$, one has
\be\label{the1:proof-induction}
x^*\in B(c_k, R_k^2-2(F(x_{k}^+)-F^*)/\alpha).
\ee
Let $y=x^*$ in \eqref{maineq} we have
$\normtwo{x^*-x^{++}}^2 \leq (1-\alpha t)\normtwo{G_{t}(x)^2}/\alpha^2-2(F(x^+)-F^*)/\alpha$, implying
\begin{equation}\label{eq1}
x^*\in B(x^{++}, \normtwo{G_{t}(x)}^2(1-\alpha t)/\alpha^2-2(F(x^+)-F^*)/\alpha).
\end{equation}
Setting $x=x_0$ in \eqref{eq1} shows that \eqref{the1:proof-induction} holds for $k=0$. We now assume that \eqref{the1:proof-induction} holds for some $k\geq 0$, and in the following we will prove that \eqref{the1:proof-induction} holds for $k+1$. Combining \eqref{the1:proof-induction} and the first inequality of \eqref{equ23} yields
\be\label{the1:proof-induction-1}
x^*\in B(c_k, R_k^2-t\normtwo{G_{t}(x_{k+1})}^2/\alpha-2(F(x_{k+1}^+)-F^*)/\alpha).
\ee
By setting $x=x_{k+1}$ in \eqref{eq1}, we get
{
\be\label{the1:proof-induction-2}
x^*\in B(x_{k+1}^{++}, \normtwo{G_{t}(x_{k+1})}^2(1-\alpha t)/\alpha^2-2(F(x_{k+1}^+)-F^*)/\alpha).
\ee
}
We now apply Lemma \ref{lem:enclosing-ball} to \eqref{the1:proof-induction-1} and \eqref{the1:proof-induction-2}. Specifically, we set $x_B=x_{k+1}^{++}$, $x_A=c_k$, $\epsilon=\alpha t$, $r_A=R_k$, $r_B=\normtwo{G_{t}(x_{k+1})}/\alpha$, $\delta=\frac{2}{\alpha}(F(x_k^+)-F^*)$, and note that $\|x_A-x_B\|^2\geq r_B^2$ because of the second inequality of \eqref{equ23}. Then Lemma \ref{lem:enclosing-ball} indicates that there exists $c_{k+1}$ such that
\be\label{the1:proof-3}
x^* \in B(c_{k+1}, (1-1/\sqrt{\kappa})R_k^2 - 2(F(x_{k+1}^+)-F^*)/\alpha),
\ee
\textit{i.e.}, \eqref{the1:proof-induction} holds for $k+1$ with $R_{k+1}^2\leq (1-\sqrt{\alpha t})R_k^2$.
Note that $c_{k+1}$ is the center of the minimum enclosing ball of the intersection of the two balls in \eqref{the1:proof-induction-1} and \eqref{the1:proof-induction-2}, and can be computed in the same way as Algorithm 1 of \cite{bubeck2015geometric}.
From \eqref{the1:proof-3} we obtain that $\|x^*-c_{k+1}\|^2\leq (1-\sqrt{\alpha t})R_k^2\leq(1-\sqrt{\alpha t})^{k+1}R_0^2$. Moreover, \eqref{the1:proof-induction-1} indicates that $F(x_{k+1}^+)-F^*\leq \frac{\alpha}{2}R_k^2\leq\frac{\alpha}{2}(1-\sqrt{\alpha t})^kR_0^2$.
\end{proof}

\section{Practical Issues}\label{sec:practical}

\subsection{GeoPG with Backtracking}
In practice, the Lipschitz constant $\beta$ may be unknown to us. In this subsection, we describe a backtracking strategy for GeoPG in which $\beta$ is not needed. From the $\beta$-smoothness of $f$, we have
\be\label{beta-smooth-inequality-prox}
f(x^+)\leq f(x)- t\inp{\nabla f(x)}{G_t(x)} +t\normtwo{G_t(x)}^2/2.
\ee
Note that the inequality \eqref{maineq} holds because of \eqref{beta-smooth-inequality-prox}, which holds when $t\in(0,1/\beta]$. If $\beta$ is unknown, we can perform backtracking on $t$ such that \eqref{beta-smooth-inequality-prox} holds, which is a common practice for proximal gradient method, \textit{e.g.}, \cite{Beck-Teboulle-2009,Goldfarb-Scheinberg-fastlinesearch2011,Nesterov-07}.
Note that the key step in our analysis of GeoPG is to guarantee that the two inequalities in \eqref{equ23} hold. According to Lemma \ref{lem:equ23-reduced}, the second inequality in \eqref{equ23} holds as long as we use Algorithm \ref{Alg:Brent} or Algorithm \ref{Alg:semi-smooth-newton} to find $x_k$, and it does not need the knowledge of $\beta$. However, the first inequality in \eqref{equ23} requires $t\leq 1/\beta$, because its proof in Lemma \ref{lem:equ23-reduced} needs \eqref{maineq}. Thus, we need to perform backtracking on $t$ until \eqref{beta-smooth-inequality-prox} is satisfied, and use the same $t$ to find $x_k$ by Algorithm \ref{Alg:Brent} or Algorithm \ref{Alg:semi-smooth-newton}. Our GeoPG with backtracking (GeoPG-B) is described in Algorithm \ref{Alg:backtracking}.
\begin{algorithm}[h]
\caption{GeoPG with Backtracking (GeoPG-B)}\label{Alg:backtracking}
\begin{algorithmic}
\REQUIRE{Parameters $\alpha$, $\gamma\in(0,1)$, $\eta\in (0,1)$, initial step size $t_0>0$ and initial point $x_0$.}
\STATE{Repeat $t_0:=\eta t_0$ until \eqref{beta-smooth-inequality-prox} holds for $t=t_0$;}
\STATE{Set $c_0=x_0^{++}$, $R_0^2=\displaystyle\frac{\normtwo{G_{t_0}(x_0)}^2}{\alpha^2}(1-\alpha t_0)$;}
\FOR{$k=1,2,\ldots$}
\IF{no backtracking was performed in the $(k-1)$-st iteration}
\STATE{Set $t_k:=t_{k-1}/\gamma$;}
\ELSE
\STATE{Set $t_k:=t_{k-1}$;}
\ENDIF
\STATE{Compute $x_k$ by Algorithm \ref{Alg:Brent} or Algorithm \ref{Alg:semi-smooth-newton} with $t=t_k$;}
\WHILE{$f(x_k^+) > f(x_k) - t_k\langle\nabla f(x_k), G_{t_k}(x_k)\rangle +\frac{t_k}{2} \normtwo{G_{t_k}(x_k)}^2$}
\STATE{Set $t_k:=\eta t_{k}$ \ (backtracking);}
\STATE{Compute $x_k$ by Algorithm \ref{Alg:Brent} or Algorithm \ref{Alg:semi-smooth-newton} with $t=t_k$;}
\ENDWHILE
\STATE{Set $x_A := x_k^{++} =x_k- G_{t_k}(x_k)/\alpha$, $R_A^2=\frac{\normtwo{G_{t_k}(x_k)}^2}{\alpha^2}(1-\alpha t_k)$;}
\STATE{Set $x_B := c_{k-1}$, $R_B^2=R_{k-1}^2-\frac{2}{\alpha}(F(x_{k-1}^+)-F(x_k^+))$;}
\STATE{Compute $B(c_k,R_k^2)$: the minimum enclosing ball of $B(x_A,R_A^2)\cap B(x_B,R_B^2)$;}
\ENDFOR
\end{algorithmic}
\end{algorithm}

Note that the sequence $\{t_k\}$ generated in Algorithm \ref{Alg:backtracking} is uniformly bounded away from $0$. This is because \eqref{beta-smooth-inequality-prox} always holds when $t_k\leq 1/\beta$. As a result, we know $t_k \geq t_{\min} := \min_{i=0,\ldots,k} t_i \geq \eta/\beta$. It is easy to see that in the $k$-th iteration of Algorithm \ref{Alg:backtracking}, $x^*$ is contained in two balls:
\[
\ba{ll}
x^*\in & B\big(c_{k-1}, R_{k-1}^2-t_k\normtwo{G_{t_k}(x_{k})}^2/\alpha-2(F(x_{k}^+)-F^*)/\alpha\big)\\
x^*\in & B\big(x_{k}^{++}, \normtwo{G_{t_k}(x_{k})}^2(1-\alpha t_k)/\alpha^2-2(F(x_{k}^+)-F^*)/\alpha\big).
\ea
\]
Therefore, we have the following convergence result for Algorithm \ref{Alg:backtracking}, whose proof is similar to that for Algorithm \ref{Alg:GeoPG}. We thus omit the proof for succinctness.
\begin{theorem}
Suppose $\{(x_k,c_k,R_k,t_k)\}$ is generated by Algorithm \ref{Alg:backtracking}. For any $k\geq 0$, one has $x^*\in B(c_k,R_k^2)$ and $R_{k+1}^2\leq (1-\sqrt{\alpha t_k})R_k^2$, and thus
$\normtwo{x^*-c_k}^2\leq \prod_{i=0}^{k}(1-\sqrt{\alpha t_i})^i   R_0^2\leq (1-\sqrt{\alpha t_{min}})^k   R_0^2.$
\end{theorem}

\subsection{GeoPG with Limited Memory}

The basic idea of GeoD is that in each iteration we maintain two balls $B(y_1,r_1^2)$ and $B(y_2, r_2^2)$ that both contain $x^*$, and then compute the minimum enclosing ball of their intersection, which is expected to be smaller than both $B(y_1,r_1^2)$ and $B(y_2, r_2^2)$. One very intuitive idea that can possibly improve the performance of GeoD is to maintain more balls from the past, because their intersection should be smaller than the intersection of two balls. This idea has been proposed by \cite{bubeck2016black} and \cite{drusvyatskiy2016optimal}. Specifically, \cite{bubeck2016black} suggested to keep all the balls from past iterations and then compute the minimum enclosing ball of their intersection. For a given bounded set $Q$, the center of its minimum enclosing ball is known as the Chebyshev center, and is defined as the solution to the following problem:
\begin{equation}\label{eqch}
\min_y \max_{x\in Q} \normtwo{y-x}^2 = \min_y \max_{x\in Q} { \normtwo{y}^2-2y^\top x+\Tr(xx^\top)}.
\end{equation}
\eqref{eqch} is not easy to solve for a general set $Q$. However, when $Q := \cap_{i=1}^m B(y_i,r_i^2)$, Beck \cite{Beck-enclosing-ball-2007} proved that the relaxed Chebyshev center (RCC) \cite{eldar2008minimax}, which is a convex quadratic program, is equivalent to \eqref{eqch}, if $m < n$. Therefore, we can solve \eqref{eqch} by solving a convex quadratic program (RCC):
\begin{equation}\label{eqrelch}
\min_y \max_{(x,\bigtriangleup)\in \Gamma}\normtwo{y}^2-2y^\top x+\Tr(\bigtriangleup) = \max_{(x,\bigtriangleup)\in \Gamma}\min_y \normtwo{y}^2-2y^\top x+\Tr(\bigtriangleup) = \max_{(x,\bigtriangleup)\in \Gamma} -\normtwo{x}^2+\Tr(\bigtriangleup),
\end{equation}
where $\Gamma=\{(x,\bigtriangleup): x\in Q ,\bigtriangleup\succeq x x^\top\}$. If $Q=\cap_{i=1}^m B(c_i,r_i^2)$, then the dual of \eqref{eqrelch} is
\begin{equation}\label{dual}
\min \normtwo{C\lambda}^2-\sum_{i=1}^{m}\lambda_i\normtwo{c_i}^2+\sum_{i=1}^{m}\lambda_ir_i^2, \ \st \ \sum_{i=1}^{m}\lambda_i = 1, \quad \lambda_i\geq 0, \; i = 1, \ldots, m,
\end{equation}
where $C=[c_1,\ldots,c_m]$ and $\lambda_i, i=1,2,\ldots,m$ are the dual variables. Beck \cite{Beck-enclosing-ball-2007} proved that the optimal solutions of \eqref{eqch} and \eqref{dual} are linked by $x^*=C\lambda^*$ if $m<n$.

Now we can give our limited-memory GeoPG algorithm (L-GeoPG) as in Algorithm \ref{Alg:L-GeoPG}. 
\begin{algorithm}[h]
\caption{L-GeoPG: Limited-memory GeoPG}\label{Alg:L-GeoPG}
\begin{algorithmic}[1]
    \REQUIRE{Parameters $\alpha$, $\beta$, memory size $m>0$ and initial point $x_0$.}
    \STATE{Set $c_0=x_0^{++}$, $r_0^2=R_0^2= \normtwo{G_t(x_0)}^2 (1-1/\kappa)/\alpha^2$, and $t=1/\beta$;}
    \FOR{$k=1,2,\ldots$}
    \STATE{Use Algorithm \ref{Alg:Brent} or \ref{Alg:semi-smooth-newton} to find $x_k$;}
    \STATE{Compute $r_k^2 = \normtwo{G_t(x_k)}^2(1-1/\kappa)/\alpha^2$;}
    \STATE{Compute $B(c_k,R_k^2)$: an enclosing ball of the intersection of $B(c_{k-1},R_{k-1}^2)$ and $Q_k:=\cap_{i=k-m+1}^k B(x_i^{++},r_i^2)$ (if $k\leq m$, then set $Q_k:=\cap_{i=1}^k B(x_i^{++},r_i^2)$). This is done by setting $c_k=C\lambda^*$, where $\lambda^*$ is the optimal solution of \eqref{dual};}
    \ENDFOR
\end{algorithmic}
\end{algorithm}

\begin{remark}
Backtracking can also be incorporated into L-GeoPG. We denote the resulting algorithm as L-GeoPG-B.
\end{remark}

L-GeoPG has the same linear convergence rate as GeoPG, as we show in Theorem \ref{thm:L-GeoPG}.
\begin{theorem}\label{thm:L-GeoPG}
Consider L-GeoPG algorithm. For any $k\geq 0$, one has $x^*\in B(c_k,R_k^2)$ and $R_{k}^2\leq (1-1/\sqrt{\kappa})R_{k-1}^2$, and thus
$\normtwo{x^*-c_k}^2\leq (1-1/\sqrt{\kappa})^k R_0^2.$
\end{theorem}
\begin{proof}
Note that $Q_k:=\cap_{i=k-m+1}^k B(x_i^{++},r_i^2)\subset B(x_{k}^{++},r_k^2)$. Thus, the minimum enclosing ball of $B(c_{k-1},R_{k-1}^2)\cap B(x_{k}^{++},r_k^2)$ is an enclosing ball of $B(c_{k-1},R_{k-1}^2)\cap Q_k$. The proof then follows from the proof of Theorem \ref{the1}, and we omit it for brevity.
\end{proof}

\section{Numerical Experiments}\label{sec:num}
In this section, we compare our GeoPG algorithm with Nesterov's accelerated proximal gradient (APG) method for solving two nonsmooth problems: linear regression and logistic regression, both with elastic net regularization. Because of the elastic net term, the strong convexity parameter $\alpha$ is known. However, we assume that $\beta$ is unknown, and implement backtracking for both GeoPG and APG, \textit{i.e.}, we test GeoPG-B and APG-B (APG with backtracking). We do not target at comparing with other efficient algorithms for solving these two problems. Our main purpose here is to illustrate the performance of this new first-order method GeoPG. Further improvement of this algorithm and comparison with other state-of-the-art methods will be a future research topic.

The initial points were set to zero. To obtain the optimal objective function value $F^*$, we ran APG-B and GeoPG-B for a sufficiently long time and the smaller function value returned by the two algorithms is selected as $F^*$. APG-B was terminated if $(F(x_k)-F^*)/F^*\leq tol$, and GeoPG-B was terminated if $(F(x_k^+)-F^*)/F^*\leq tol$, where $tol=10^{-8}$ is the accuracy tolerance. The parameters used in backtracking were set to $\eta=0.5$ and $\gamma=0.9$. In GeoPG-B, we used Algorithm \ref{Alg:semi-smooth-newton} to find $x_k$, because we found that the performance of Algorithm \ref{Alg:semi-smooth-newton} is slightly better than Algorithm \ref{Alg:Brent} in practice.
The codes were written in Matlab and run on a standard PC with 3.20 GHz I5 Intel microprocessor and 16GB of memory. In all figures we reported, the $x$-axis denotes the CPU time (in seconds) and $y$-axis denotes $(F(x_k^+)-F^*)/F^*$.

\subsection{Linear regression with elastic net regularization}
In this subsection, we compare GeoPG-B and APG-B for solving linear regression with elastic net regularization, a popular problem in  machine learning and statistics \cite{Zou-Hastie-elastic-net-2005}:
\be\label{lasso-en}
\min_{x\in\br^n} \ \frac{1}{2p} \normtwo{Ax-b}^2 +\frac{\alpha}{2} \normtwo{x}^2 + \mu \normone{x},
\ee
where $A\in \br^{p\times n}$, $b\in\br^p$, $\alpha,\mu>0$ are the weighting parameters.

We conducted tests on two real datasets downloaded from the LIBSVM repository: a9a, RCV1.
The results are reported in Figure \ref{figure_leq}. In particular, we tested $\alpha=10^{-8} $ and $\mu=10^{-3},10^{-4},10^{-5}$. Note that since $\alpha$ is very small, the problems are very likely to be ill-conditioned. We see from Figure \ref{figure_leq} that GeoPG-B is faster than APG-B on these real datasets, which indicates that GeoPG-B is preferable than APG-B. In the appendix, we show more numerical results on different $\alpha$, which further confirm that GeoPG-B is faster than APG-B when the problems are more ill-conditioned.

\begin{figure}[ht]
\begin{center}
\minipage{0.5\textwidth}
 \subfigure[Dataset a9a ]{
\centerline{\includegraphics[width=0.8\linewidth]{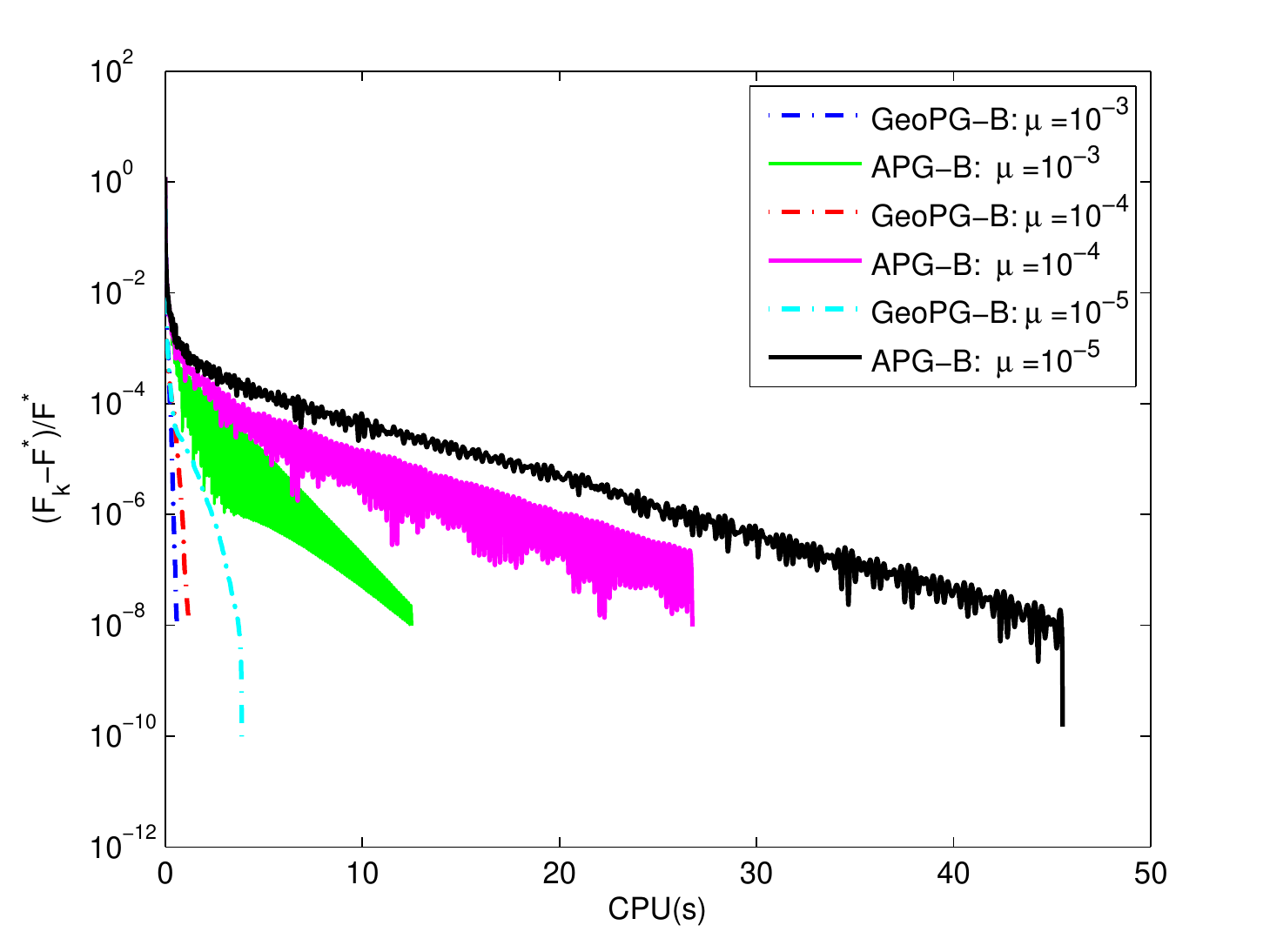}}}
\endminipage\hfill
 \minipage{0.5\textwidth}
  \subfigure[Dataset RCV1 ]{\includegraphics[width=0.8\linewidth]{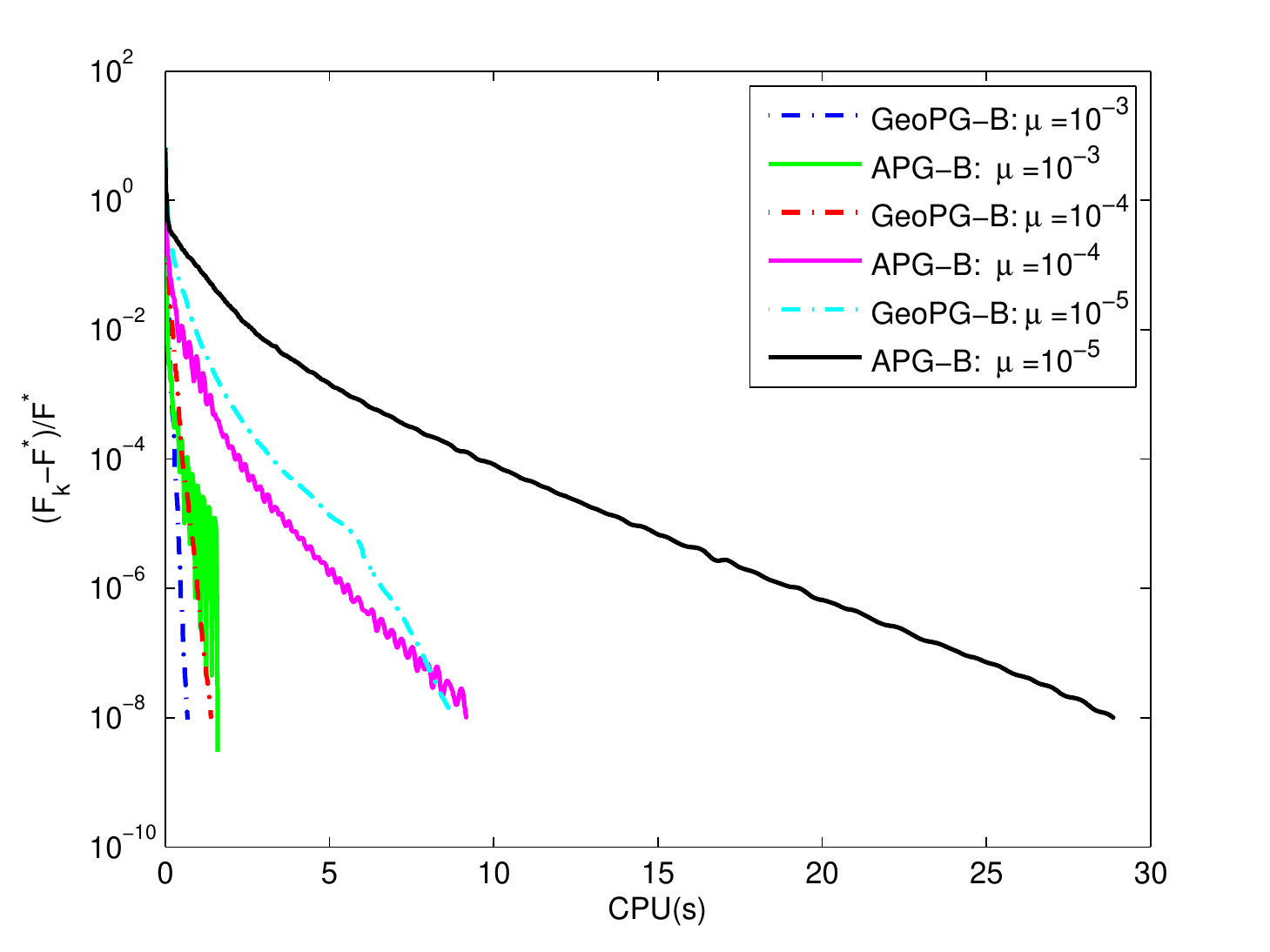}}
  \endminipage\hfill
\caption{GeoPG-B and APG-B for solving \eqref{lasso-en} with $\alpha=10^{-8}$.}
\label{figure_leq}
\end{center}
\end{figure}

\subsection{Logistic regression with elastic net regularization}

In this subsection, we compare the performance of GeoPG-B and APG-B for solving the following logistic regression problem with elastic net regularization:
\be\label{LR}
\min_{x\in\br^n} \ \frac{1}{p} \sum_{i=1}^{p} \log \big(1+\exp(-b_i \cdot a_i^\top x)\big) +\frac{\alpha}{2} \normtwo{x}^2 +\mu\normone{x},
\ee
where $a_i \in \br^n$ and $b_i \in \{\pm 1\}$ are the feature vector and class label of the $i$-th sample, respectively,
and $\alpha$, $\mu > 0$ are the weighting parameters.

We tested GeoPG-B and APG-B for solving \eqref{LR} on the three real datasets a9a, RCV1 and Gisette from LIBSVM, and the results are reported in Figure \ref{figure_log}. In particular, we tested $\alpha=10^{-8}$ and $\mu=10^{-3},10^{-4},10^{-5}$. Figure \ref{figure_log} shows that for the same $\mu$, GeoPG-B is much faster than APG-B. More numerical results are provided in the appendix, which also indicate that GeoPG-B is much faster than APG-B, especially when the problems are more ill-conditioned.

\begin{figure}[ht]
\begin{center}
\includegraphics[width=0.32\linewidth]{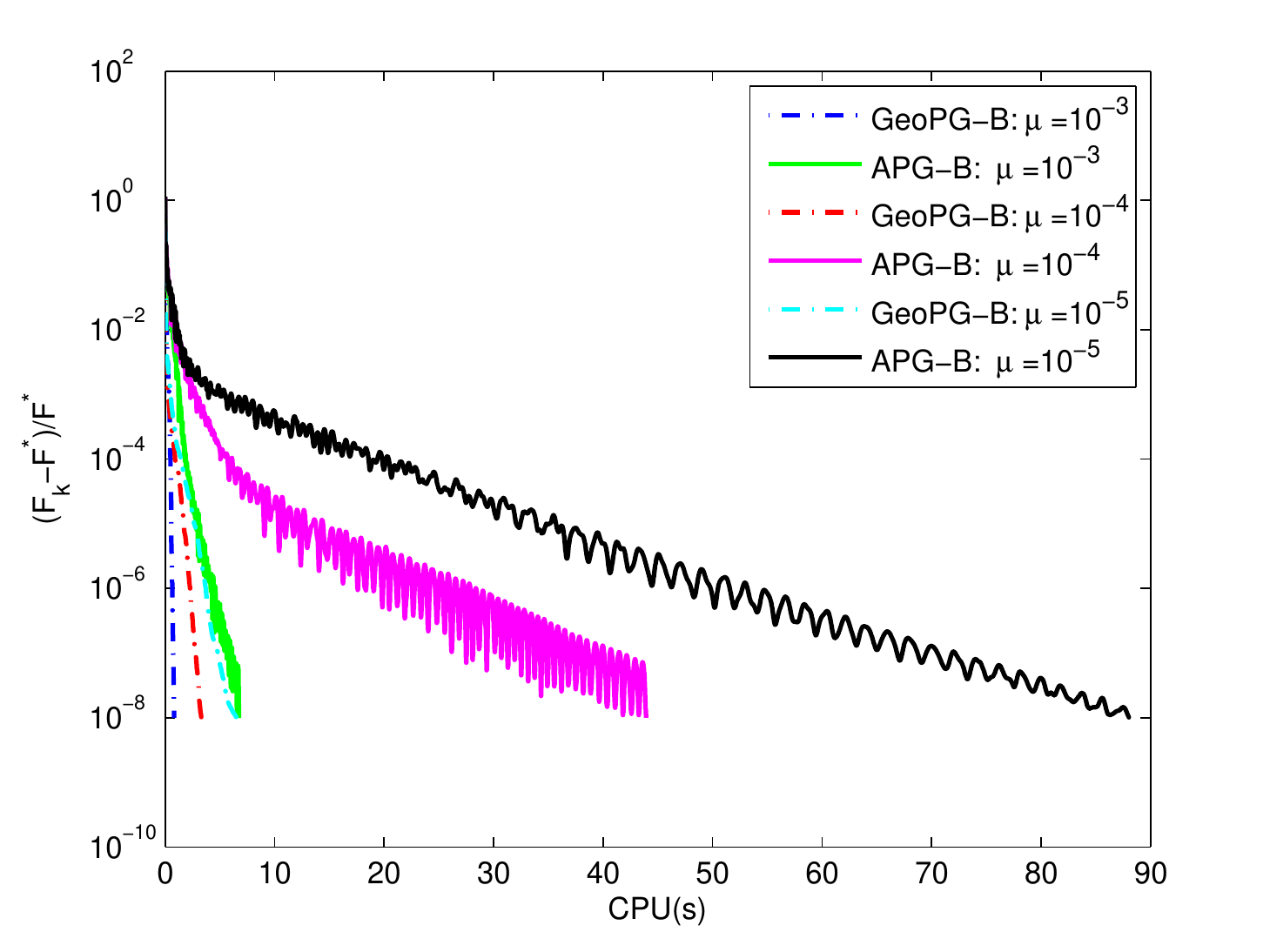}
\includegraphics[width=0.32\linewidth]{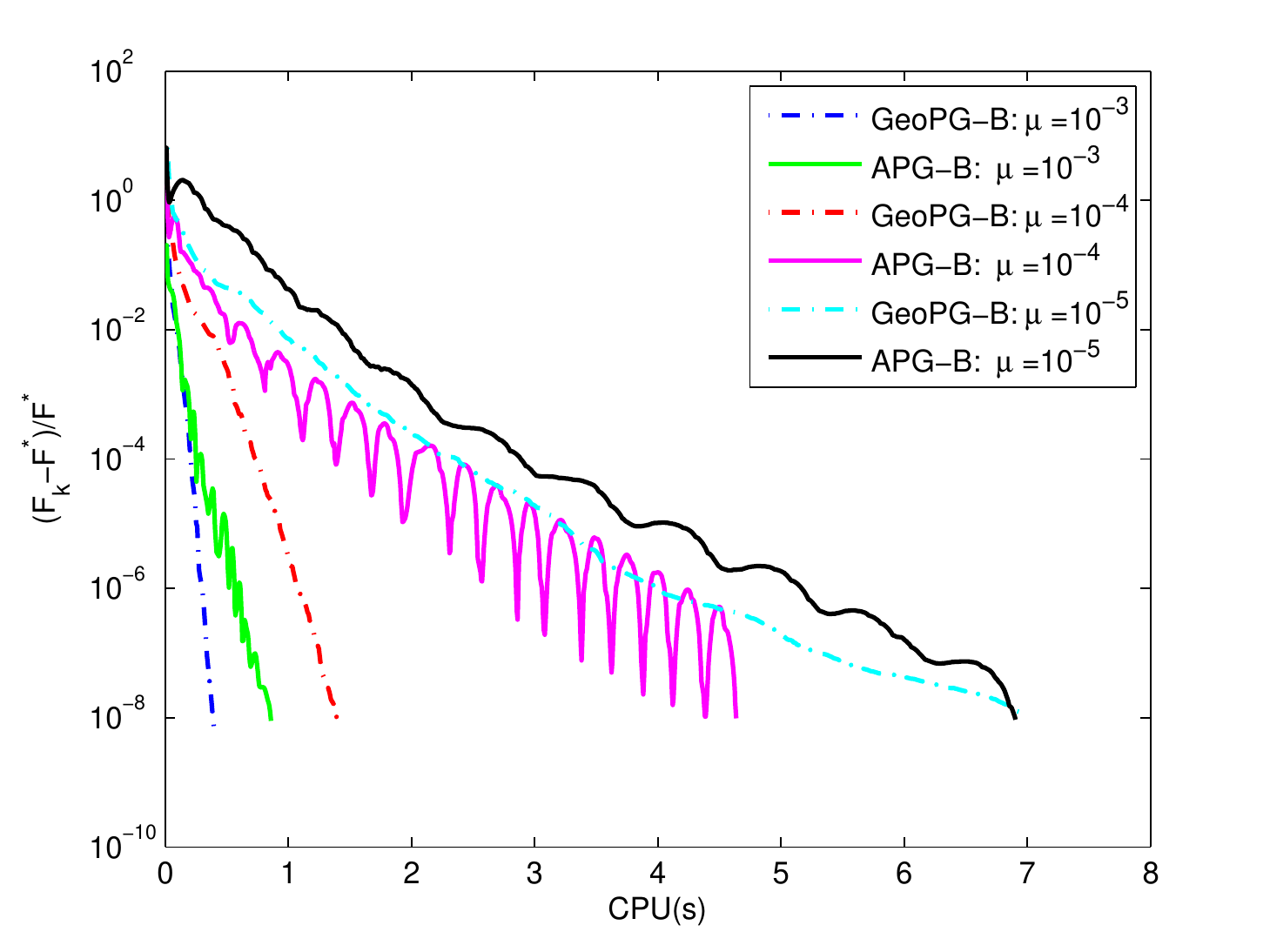}
\includegraphics[width=0.32\linewidth]{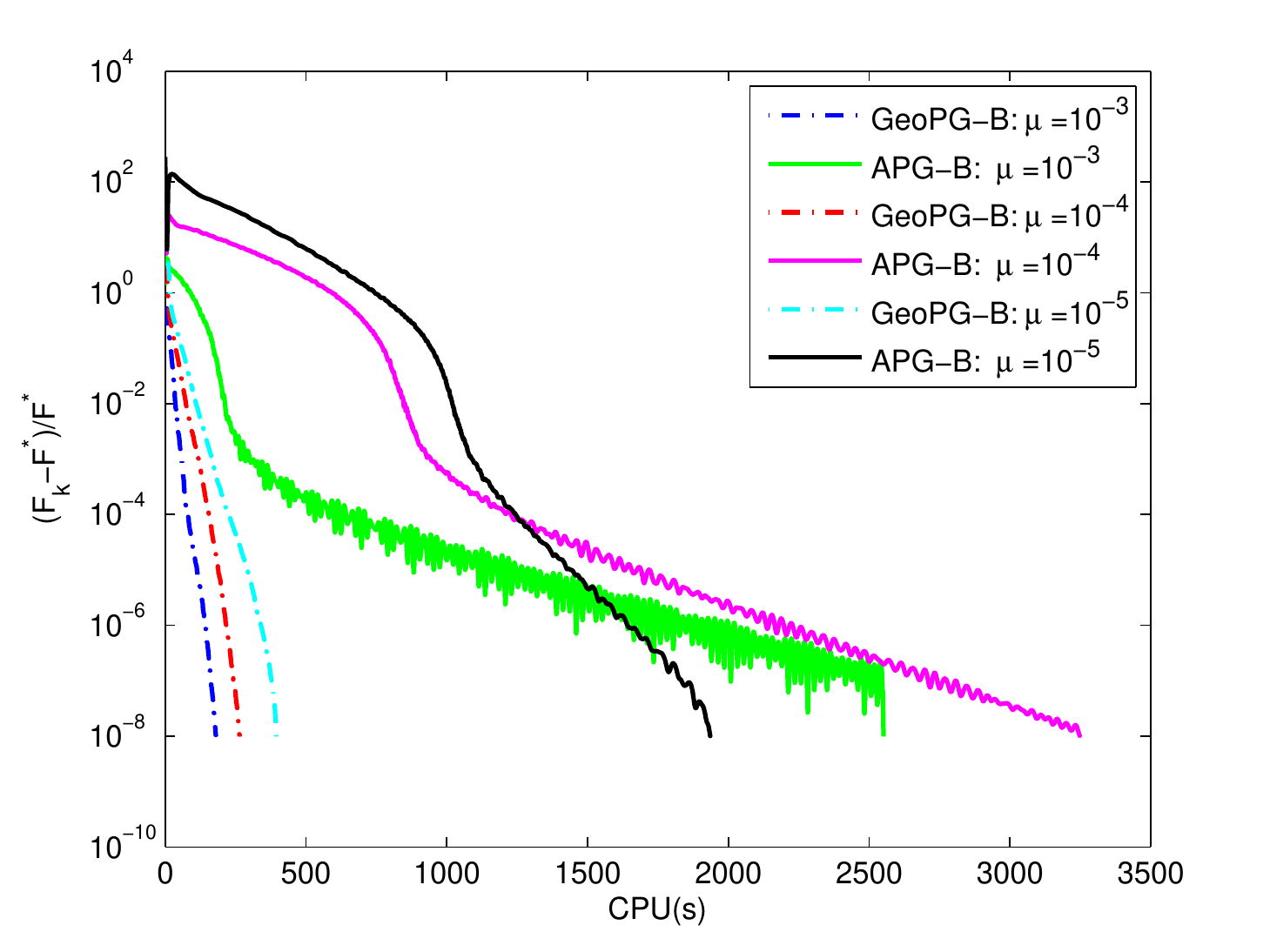}
\caption{GeoPG-B and APG-B for solving \eqref{LR} with $\alpha=10^{-8}$. Left: dataset a9a; Middle: dataset RCV1; Right: dataset Gisette.}
\label{figure_log}
\end{center}
\end{figure}

\subsection{Numerical results of L-GeoPG-B}

In this subsection, we test GeoPG with limited memory described in Algorithm \ref{Alg:L-GeoPG} for solving \eqref{LR} on the Gisette dataset. Since we still need to use the backtracking technique, we actually tested L-GeoPG-B. The results with different memory sizes $m$ are reported in Figure \ref{figure_mem}. Note that $m=0$ corresponds to the original GeoPG-B without memory. The subproblem \eqref{dual} is solved using the function ``quadprog'' in Matlab.
From Figure \ref{figure_mem} we see that roughly speaking, L-GeoPG-B performs better for larger memory sizes, and in most cases, the performance of L-GeoPG-B with $m=100$ is the best among the reported results. This indicates that the limited-memory idea indeed helps improve the performance of GeoPG.
\begin{figure}[ht]
\begin{center}
\includegraphics[width=0.32\linewidth]{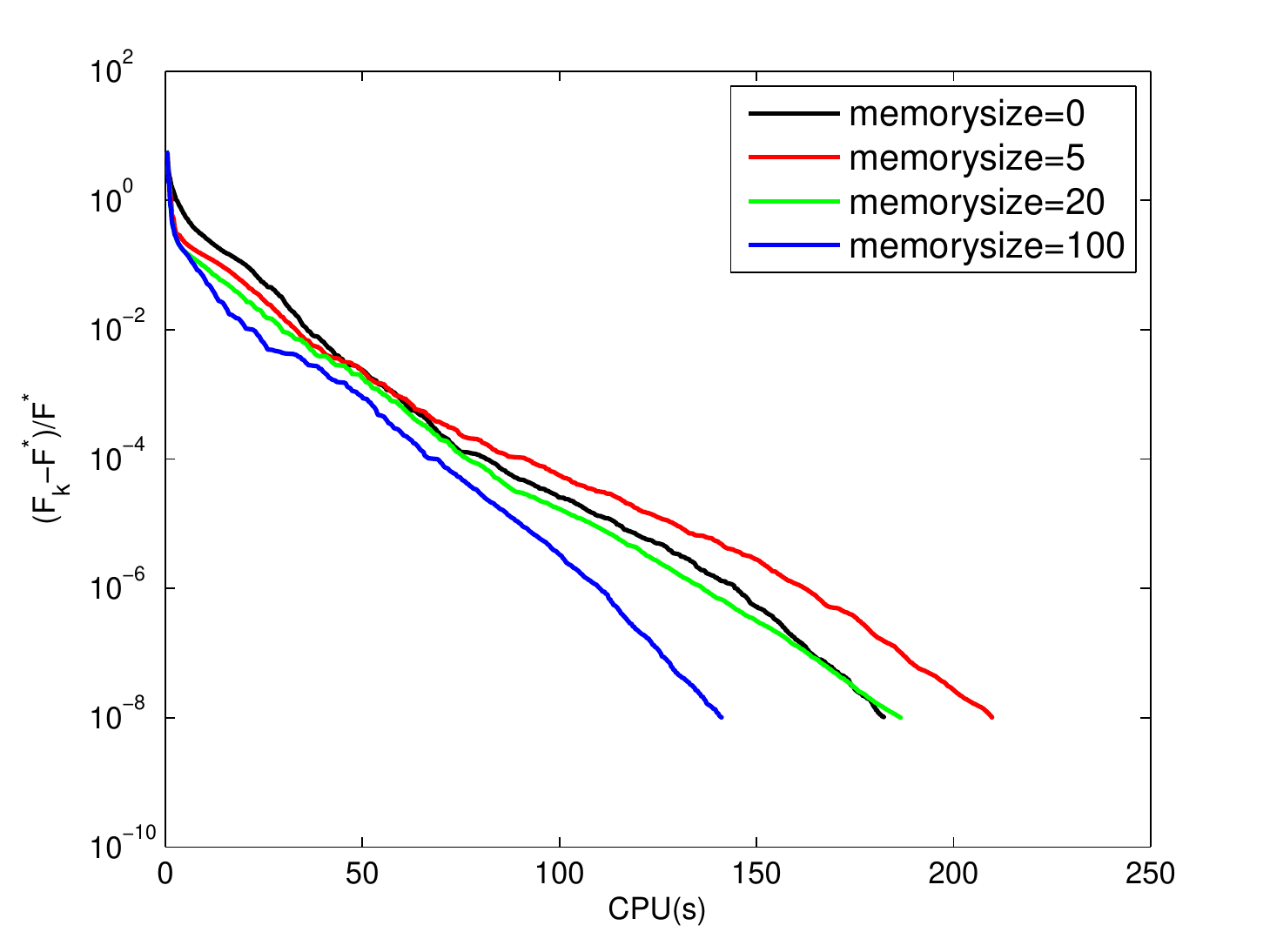}
\includegraphics[width=0.32\linewidth]{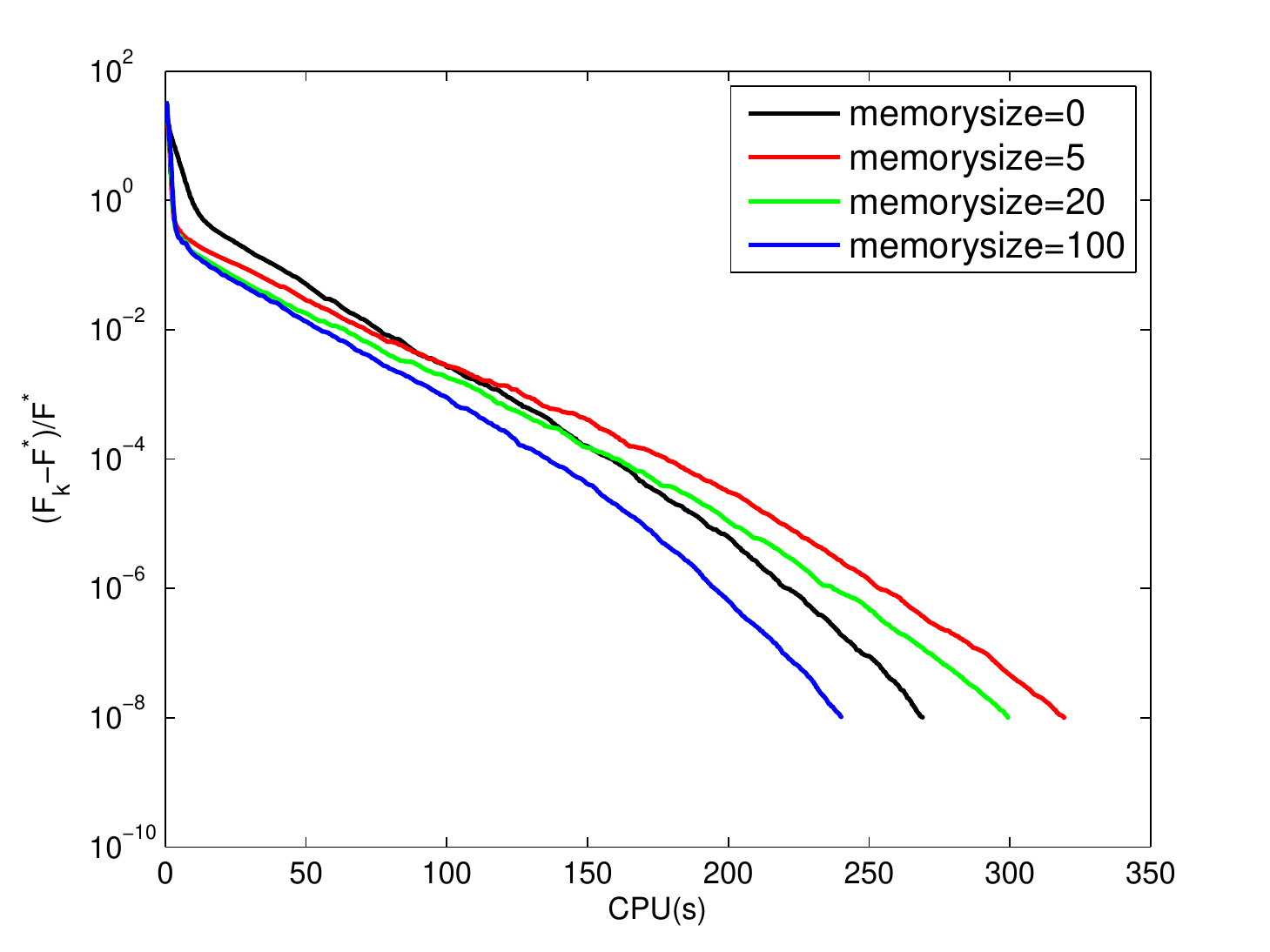}
\includegraphics[width=0.32\linewidth]{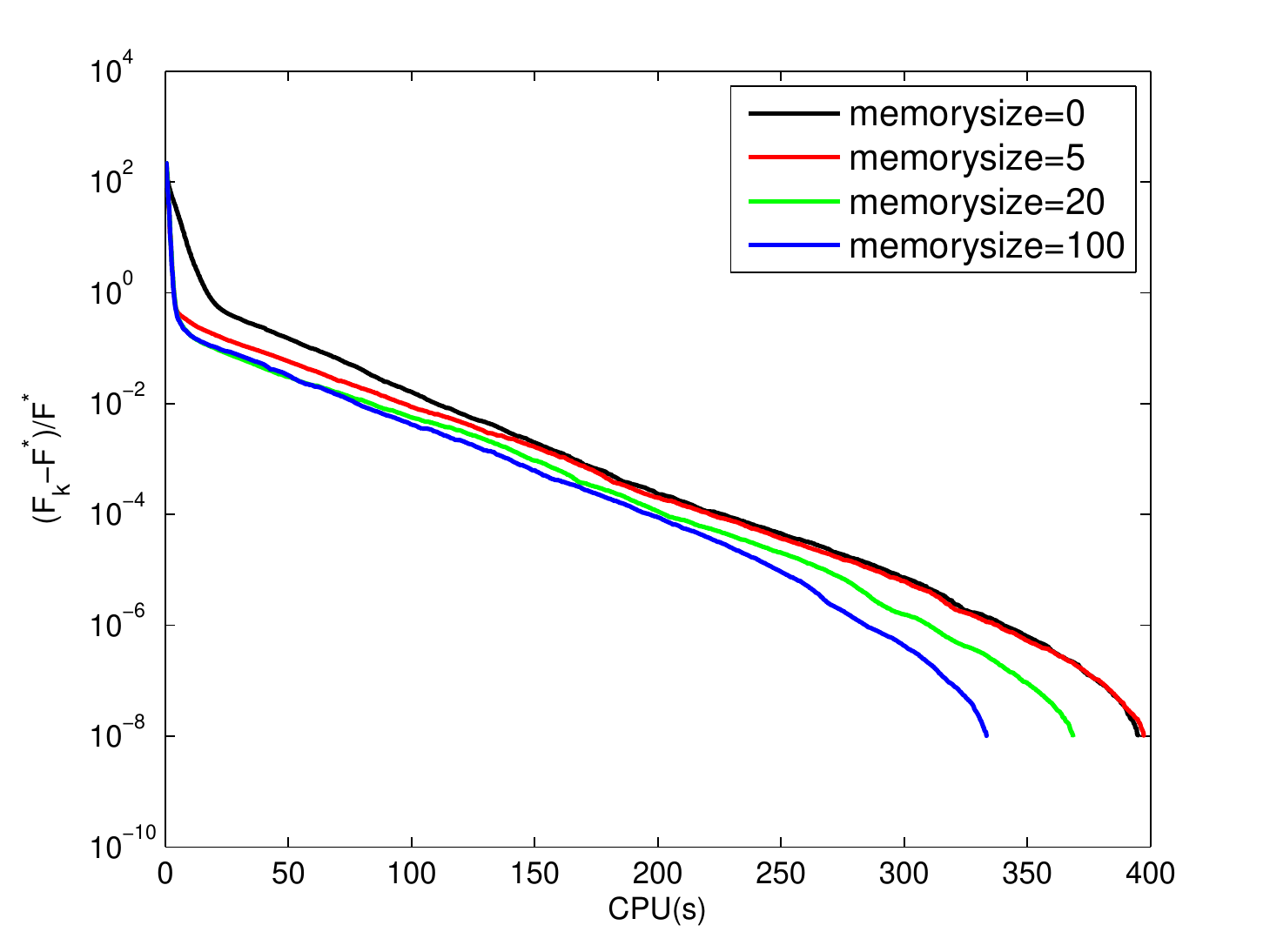}
\caption{L-GeoPG-B for solving \eqref{LR} on the dataset Gisette with $\alpha=10^{-8}$. Left: $\lambda = 10^{-3}$; Middle: $\lambda = 10^{-4}$; Right: $\lambda = 10^{-5}$.}
\label{figure_mem}
\end{center}
\end{figure}

\section{Conclusions}\label{sec:conclusion}
In this paper, we proposed a GeoPG algorithm for solving nonsmooth convex composite problems, which is an extension of the recent method GeoD that can only handle smooth problems. We proved that GeoPG enjoys the same optimal rate as Nesterov's accelerated gradient method for solving strongly convex problems. The backtracking technique was adopted to deal with the case when the Lipschitz constant is unknown. Limited-memory GeoPG was also developed to improve the practical performance of GeoPG. Numerical results on linear regression and logistic regression with elastic net regularization demonstrated the efficiency of GeoPG. It would be interesting to see how to extend GeoD and GeoPG to tackle non-strongly convex problems, and how to further accelerate the running time of GeoPG. We leave these questions in future work.

\newpage
\bibliography{All}

\begin{thebibliography}{21}
\providecommand{\natexlab}[1]{#1}
\providecommand{\url}[1]{\texttt{#1}}
\expandafter\ifx\csname urlstyle\endcsname\relax
  \providecommand{\doi}[1]{doi: #1}\else
  \providecommand{\doi}{doi: \begingroup \urlstyle{rm}\Url}\fi

\bibitem[Bubeck et~al.(2015)Bubeck, Lee, and Singh]{bubeck2015geometric}
S.~Bubeck, Y.-T. Lee, and M.~Singh.
\newblock A geometric alternative to {N}esterov's accelerated gradient descent.
\newblock \emph{arXiv preprint arXiv:1506.08187}, 2015.

\bibitem[Nesterov(1983)]{Nesterov-1983}
Y.~E. Nesterov.
\newblock A method for unconstrained convex minimization problem with the rate
  of convergence $\mathcal{O}(1/k^2)$.
\newblock \emph{Dokl. Akad. Nauk SSSR}, 269:\penalty0 543--547, 1983.

\bibitem[Nesterov(2004)]{NesterovConvexBook2004}
Y.~E. Nesterov.
\newblock \emph{Introductory lectures on convex optimization: A basic course}.
\newblock Applied Optimization. Kluwer Academic Publishers, Boston, MA, 2004.
\newblock ISBN 1-4020-7553-7.

\bibitem[Su et~al.(2014)Su, Boyd, and Cand\`es]{Su-Boyd-Candes-NIPS-2014}
W.~Su, S.~Boyd, and E.~J. Cand\`es.
\newblock A differential equation for modeling {N}esterov's accelerated
  gradient method: Theory and insights.
\newblock In \emph{NIPS}, 2014.

\bibitem[Attouch et~al.(2016)Attouch, Chbani, Peypouquet, and
  Redont]{Attouch-2016-MP}
H.~Attouch, Z.~Chbani, J.~Peypouquet, and P.~Redont.
\newblock Fast convergence of inertial dynamics and algorithms with asymptotic
  vanishing viscosity.
\newblock \emph{Mathematical Programming}, 2016.

\bibitem[Lessard et~al.(2016)Lessard, Recht, and
  Packard]{Lessard-Recht-Packard-siopt-2016}
L.~Lessard, B.~Recht, and A.~Packard.
\newblock Analysis and design of optimization algorithms via integral quadratic
  constraints.
\newblock \emph{SIAM Journal on Optimization}, 26\penalty0 (1):\penalty0
  57--95, 2016.

\bibitem[Wibisono et~al.(2016)Wibisono, Wilson, and Jordan]{Wibisono-PNAS-2016}
A.~Wibisono, A.~Wilson, and M.~I. Jordan.
\newblock A variational perspective on accelerated methods in optimization.
\newblock \emph{Proceedings of the National Academy of Sciences}, 133:\penalty0
  E7351--E7358, 2016.

\bibitem[Bland et~al.(1981)Bland, Goldfarb, and
  Todd]{bland-goldfarb-todd-ellipsoid-survey}
R.~G. Bland, D.~Goldfarb, and M.~J. Todd.
\newblock The ellipsoid method: A survey.
\newblock \emph{Operations Research}, 29:\penalty0 1039--1091, 1981.

\bibitem[Bubeck and Lee(2016)]{bubeck2016black}
S.~Bubeck and Y.-T. Lee.
\newblock Black-box optimization with a politician.
\newblock \emph{ICML}, 2016.

\bibitem[Drusvyatskiy et~al.(2016)Drusvyatskiy, Fazel, and
  Roy]{drusvyatskiy2016optimal}
D.~Drusvyatskiy, M.~Fazel, and S.~Roy.
\newblock An optimal first order method based on optimal quadratic averaging.
\newblock \emph{SIAM Journal on Optimization}, 2016.

\bibitem[Brent(1973)]{brent2013algorithms}
R.~P. Brent.
\newblock An algorithm with guaranteed convergence for finding a zero of a
  function.
\newblock In \emph{Algorithms for Minimization without Derivatives}. Englewood
  Cliffs, NJ: Prentice-Hall, 1973.

\bibitem[Dekker(1969)]{Dekker-1969}
T.~J. Dekker.
\newblock Finding a zero by means of successive linear interpolation.
\newblock In \emph{Constructive Aspects of the Fundamental Theorem of Algebra}.
  London: Wiley-Interscience, 1969.

\bibitem[Gerdts et~al.(2017)Gerdts, Horn, and Kimmerle]{Gerdts2017}
M.~Gerdts, S.~Horn, and S.~Kimmerle.
\newblock Line search globalization of a semismooth {N}ewton method for
  operator equations in {H}ilbert spaces with applications in optimal control.
\newblock \emph{Journal of Industrial And Management Optimization}, 13\penalty0
  (1):\penalty0 47--62, 2017.

\bibitem[Hans and Raasch(2015)]{hans2015global}
E.~Hans and T.~Raasch.
\newblock Global convergence of damped semismooth {N}ewton methods for {L}1
  {T}ikhonov regularization.
\newblock \emph{Inverse Problems}, 31\penalty0 (2):\penalty0 025005, 2015.

\bibitem[Beck and Teboulle(2009)]{Beck-Teboulle-2009}
A.~Beck and M.~Teboulle.
\newblock A fast iterative shrinkage-thresholding algorithm for linear inverse
  problems.
\newblock \emph{SIAM J. Imaging Sciences}, 2\penalty0 (1):\penalty0 183--202,
  2009.

\bibitem[Scheinberg et~al.(2014)Scheinberg, Goldfarb, and
  Bai]{Goldfarb-Scheinberg-fastlinesearch2011}
K.~Scheinberg, D.~Goldfarb, and X.~Bai.
\newblock Fast first-order methods for composite convex optimization with
  backtracking.
\newblock \emph{Foundations of Computational Mathematics}, 14\penalty0
  (3):\penalty0 389--417, 2014.

\bibitem[Nesterov(2013)]{Nesterov-07}
Y.~E. Nesterov.
\newblock Gradient methods for minimizing composite functions.
\newblock \emph{Mathematical Programming}, 140\penalty0 (1):\penalty0 125--161,
  2013.

\bibitem[Beck(2007)]{Beck-enclosing-ball-2007}
A.~Beck.
\newblock On the convexity of a class of quadratic mappings and its application
  to the problem of finding the smallest ball enclosing a given intersection of
  balls.
\newblock \emph{Journal of Global Optimization}, 39\penalty0 (1):\penalty0
  113--126, 2007.

\bibitem[Eldar et~al.(2008)Eldar, Beck, and Teboulle]{eldar2008minimax}
Y.~C. Eldar, A.~Beck, and M.~Teboulle.
\newblock A minimax {C}hebyshev estimator for bounded error estimation.
\newblock \emph{IEEE Transactions on Signal Processing}, 56\penalty0
  (4):\penalty0 1388--1397, 2008.

\bibitem[Zou and Hastie(2005)]{Zou-Hastie-elastic-net-2005}
H.~Zou and T.~Hastie.
\newblock Regularization and variable selection via the elastic net.
\newblock \emph{Journal of the Royal Statistical Society, Series B},
  67\penalty0 (2):\penalty0 301--320, 2005.

\bibitem[Lee et~al.(2014)Lee, Sun, and Saunders]{lee2014proximal}
J.~D. Lee, Y.~Sun, and M.~A. Saunders.
\newblock Proximal {N}ewton-type methods for minimizing composite functions.
\newblock \emph{SIAM Journal on Optimization}, 24\penalty0 (3):\penalty0
  1420--1443, 2014.

\end{thebibliography}
\bibliographystyle{unsrtnat}

\newpage
\appendix

{\Large\bf Appendix}

{
\section{Geometric Interpretation of GeoPG}
We argue that the geometric intuition of GeoPG is still clear. Note that we are still constructing two balls that contain $x^*$ and shrink at the same absolute amount. In GeoPG, since we assume that the smooth function $f$ is strongly convex, we naturally have one ball that contains $x^*$, and this ball is related to the proximal gradient $G_t$, instead of the gradient due to the presence of the nonsmooth function $h$. To construct the other ball, GeoD needs to perform an exact line search, while our GeoPG needs to find the root of a newly constructed function $\bar{\phi}$, which is again due to the presence of the nonsmooth function $h$. The two changes of GeoPG from GeoD are: replace gradient by proximal gradient; replace the exact line search by finding the root of $\bar{\phi}$, both of which are resulted by the presence of the nonsmooth function $h$.
}

\section{Proofs}

\subsection{Proof of Lemma 3.1}
\begin{proof}
From the $\beta$-smoothness of $f$, we have
\be\label{beta-smooth-inequality-prox}
f(x^+)\leq f(x)- t\inp{\nabla f(x)}{G_t(x)} +\frac{t}{2}\normtwo{G_t(x)}^2.
\ee
Combining \eqref{beta-smooth-inequality-prox} with
\begin{equation}\label{sc-lower-bound}
f(x) + \inp{\nabla f(x)}{y-x} +\frac{\alpha}{2}\normtwo{y-x}^2 \leq f(y), \ \forall x,y\in\br^n,
\end{equation}
yields that
\begin{equation}\label{maineq}
\begin{aligned}
F(x^+)&\leq f(y)-\inp{\nabla f(x)}{y-x} -\frac{\alpha}{2}\normtwo{y-x}^2- t\inp{\nabla f(x)}{G_t(x)} +\frac{t}{2}\normtwo{G_t(x)}^2 + h(x^+)\\
&=F(y) -\frac{\alpha}{2}\normtwo{y-x}^2+\frac{t}{2}\normtwo{G_t(x)}^2 + h(x^+)-h(y)-\inp{\nabla f(x)-G_t(x)}{y-x^+}-\inp{G_t(x)}{y-x^+}\\
&\leq F(y) -\frac{\alpha}{2}\normtwo{y-x}^2+\frac{t}{2}\normtwo{G_t(x)}^2 -\inp{G_t(x)}{y-x^+},
\end{aligned}
\end{equation}
where the last inequality is due to the convexity of $h$ and
$G_t(x)\in \nabla f(x) + \partial h(x^+)$.
\end{proof}

\subsection{Proof of Lemma 3.2}
\begin{proof}
Assume
\be\label{equ23-reduced}
\langle x_k^+ - x_k, x_{k-1}^+-x_k \rangle \leq 0, \mbox{ and } \langle x_k^+ - x_k, x_k - c_{k-1}\rangle \geq 0,
\ee
holds. By letting $y=x_{k-1}^+$ and $x=x_k$ in \eqref{maineq}, we have
\begin{align*}
F(x_k^+) &\leq F(x_{k-1}^+) -\inp{G_t(x_k)}{x_{k-1}^+-x_k}-\frac{t}{2}\normtwo{G_t(x_k)}^2-\frac{\alpha}{2}\normtwo{x_{k-1}^+-x_k}^2\\
&= F(x_{k-1}^+) +\frac{1}{t}\inp{x_k^+ -x_k}{x_{k-1}^+-x_k}-\frac{t}{2}\normtwo{G_t(x_k)}^2-\frac{\alpha}{2}\normtwo{x_{k-1}^+-x_k}^2\\
& \leq F(x_{k-1}^+)-\frac{t}{2}\normtwo{G_t(x_k)}^2,
\end{align*}
where the last inequality is due to \eqref{equ23-reduced}. Moreover, from the definition of $x_k^{++}$ and \eqref{equ23-reduced} it is easy to see
\begin{align*}
\normtwo{x_k^{++}-c_{k-1}}^2&=\normtwo{x_k-c_{k-1}}^2 +\frac{2}{\alpha t} \inp{x_k^+ -x_k}{x_{k}-c_{k-1}}+\frac{1}{\alpha^2}\normtwo{G_t(x_k)}^2 \geq \frac{1}{\alpha^2}\normtwo{G_t(x_k)}^2.
\end{align*}
\end{proof}

\subsection{Proof of Lemma 3.3}
Before we prove Lemma 3.3, we need the following well-know result, which can be found in \cite{lee2014proximal}.

\noindent{\bf Lemma.} (see Lemma 3.9 of \cite{lee2014proximal})
For $t\in(0,1/\beta]$, $G_t(x)$ is strongly monotone, i.e.,
\be\label{lemma-lee}
 \inp{G_t(x)- G_t(y)}{x-y} \geq \frac{\alpha}{2} \normtwo{x-y}^2, \forall x,y.
\ee

We are now ready to prove Lemma 3.3.
\begin{proof}
We prove (i) first.
\begin{align*}
\abs{\phi_{t,x,c}(z_1)-\phi_{t,x,c}(z_2)}&=\abs{\inp{z_1^+-z_1 -(z_2^+-z_2)}{x-c}} \leq \normtwo{z_1^+-z_2^+-(z_1-z_2)}\normtwo{x-c}\\
&\leq (\normtwo{\text{prox}_{th}(z_1-t\nabla f(z_1))-\text{prox}_{th}(z_2-t\nabla f(z_2))}+\normtwo{z_1-z_2})\normtwo{x-c}\\
&\leq (2+t\beta)\normtwo{x-c}\normtwo{z_1-z_2},
\end{align*}
where the last inequality is due to the non-expansiveness of the proximal mapping operation.

We now prove (ii). For $s_1<s_2$, let $z_1 = x+s_1(c-x)$ and $z_2= x+s_2(c-x)$. We have
\begin{align*}
\bar\phi_{t,x,c}(s_2)-\bar\phi_{t,x,c}(s_1)&=\inp{z_2^+-z_2 -(z_1^+-z_1)}{x-c} = \frac{t}{s_2-s_1}\inp{G_t(z_2)-G_t(z_1)}{z_2-z_1}\\
&\geq \frac{\alpha t}{2}(s_2-s_1)\normtwo{x-c}^2>0,
\end{align*}
where the first inequality follows from \eqref{lemma-lee}.
\end{proof}

\section{Numerical Experiment on Other Datasets}\label{sec:num}
In this section, we report some numerical results of other data sets. Here  we set the terminate condition as $\|G_t(x_k^+)\|_\infty\leq tol$ for GeoP-B and $\|G_t(x_k)\|_\infty\leq tol$  for APG-B.

\subsection{Linear regression with elastic net regularization}
In this subsection, we compare GeoPG-B and APG-B for solving linear regression with elastic net regularization:
\be\label{lasso-en}
\min_{x\in\br^n} \ \frac{1}{2p} \normtwo{Ax-b}^2 +\frac{\alpha}{2} \normtwo{x}^2 + \mu \normone{x},
\ee
where $A\in \br^{p\times n}$, $b\in\br^p$, $\alpha,\mu>0$ are weighting parameters.

We first compare these two algorithms on some synthetic data. In our experiments, entries of $A$ were drawn randomly from the standard Gaussian distribution, the solution $\bar{x}$ was a sparse vector with 10\% nonzero entries whose locations are uniformly random and whose values follow the Gaussian distribution $3*\mathcal{N}(0,1)$, and $b = A*\bar{x} + \mathbf{n}$, where the noise $\mathbf{n}$ follows the Gaussian distribution $0.01*\mathcal{N}(0,1)$. Moreover, since we assume that the strong convexity parameter of \eqref{lasso-en} is equal to $\alpha$, when $p>n$, we manipulate $A$ such that the smallest eigenvalue of $A^\top A$ is equal to 0. Specifically, when $p>n$, we truncate the smallest eigenvalue of $A^\top A$ to 0, and obtain the new $A$ by eigenvalue decomposition of $A^\top A$. We set $tol = 10^{-8}$.

In Tables \ref{tab:lasso-1}, \ref{tab:lasso-2} and \ref{tab:lasso-3}, we report the comparison results of GeoPG-B and APG-B for solving different instances of \eqref{lasso-en}. We use ``f-ev'', ``g-ev'', ``p-ev'' and ``MVM'' to denote the number of evaluations of objective function, gradient, proximal mapping of $\ell_1$ norm, and matrix-vector multiplications, respectively. The CPU times are in seconds. We use ``--'' to denote that the algorithm does not converge in $10^5$ iterations. We tested different values of $\alpha$, which reflect different condition numbers of the problem. We also tested different values of $\mu$, which was set to $\mu = (10^{-3},10^{-4},10^{-5})/p\times \|A^\top b\|_\infty$, respectively. ``f-diff'' denotes the absolute difference of the objective values returned by the two algorithms.

From Tables \ref{tab:lasso-1}, \ref{tab:lasso-2} and \ref{tab:lasso-3} we see that GeoPG-B is more efficient than APG-B in terms of CPU time when $\alpha$ is small. For example, Table \ref{tab:lasso-1} indicates that GeoPG-B is faster than APG-B when $\alpha\leq 10^{-4}$, Table \ref{tab:lasso-2} indicates that GeoPG-B is faster than APG-B when $\alpha\leq 10^{-6}$, and Table \ref{tab:lasso-3} shows that GeoPG-B is faster than APG-B when $\alpha\leq 10^{-8}$. Since a small $\alpha$ corresponds to a large condition number, we can conclude that in this case GeoPG-B is more preferable than APG-B for ill-conditioned problems. Note that ``f-diff'' is very small in all cases, which indicates that the solutions returned by GeoPG-B and APG-B are very close.

We also conducted tests on three real datasets downloaded from the LIBSVM repository: a9a, RCV1 and Gisette, among which a9a and RCV1 are sparse and Gisette is dense. The size and sparsity (percentage of nonzero entries) of these three datasets are $(32561\times 123, 11.28\%)$, $(20242\times 47236, 0.16\%)$ and $(6000\times 5000, 99.1\%)$, respectively. The results are reported in Tables \ref{tab: leq a9a}, \ref{tab: leq rcv1} and \ref{tab: leq gisette}, where $\alpha=10^{-2}, 10^{-4}, 10^{-6}, 10^{-8}, 10^{-10}$ and $\mu=10^{-3}, 10^{-4}, 10^{-5}$. We see from these tables that GeoPG-B is faster than APG-B on these real datasets when $\alpha$ is small, i.e., when the problem is more ill-conditioned.

\subsection{Logistic regression with elastic net regularization}

In this subsection, we compare the performance of GeoPG-B and APG-B for solving the following logistic regression problem with elastic net regularization:
\be\label{LR}
\min_{x\in\br^n} \ \frac{1}{p} \sum_{i=1}^{p} \log (1+\exp(-b_i \cdot a_i^\top x)) +\frac{\alpha}{2} \normtwo{x}^2 +\mu\normone{x},
\ee
where $a_i \in \br^n$ and $b_i \in \{\pm 1\}$ are the feature vector and class label of the $i$-th sample, respectively,
and $\alpha$, $\mu > 0$ are weighting parameters.

We first compare GeoPG-B and APG-B for solving \eqref{LR} on some synthetic data. In our experiments, each $a_i$ was drawn randomly from the standard Gaussian distribution, the linear model parameter $\bar{x}$ was a sparse vector with 10\% nonzero entries whose locations are uniformly random and whose values follow the Gaussian distribution $3*\mathcal{N}(0,1)$, and $\ell = A*\bar{x} + \mathbf{n}$, where noise $\mathbf{n}$ follows the Gaussian distribution $0.01*\mathcal{N}(0,1)$. Then, we generate class labels as bernoulli random variables with the parameter $1/(1+\exp{\ell_i})$. We set $tol = 10^{-8}$.

In Tables \ref{tab:LR-Syn-1}, \ref{tab:LR-Syn-2} and \ref{tab:LR-Syn-3} we report the comparison results of GeoPG-B and APG-B for solving different instances of \eqref{LR}. From results in these tables we again observe that GeoPG-B is faster than APG-B when $\alpha$ is small, i.e., when the condition number is large.

We also tested GeoPG-B and APG-B for solving \eqref{LR} on the three real datasets a9a, RCV1 and Gisette from LIBSVM, and the results are reported in Tables \ref{tab: log a9a}, \ref{tab: log rcv1} and \ref{tab: log gisette}. We again have the similar observations as before, i.e., GeoPG-B is faster than APG-B for more ill-conditioned problems.

{
\subsection{More discussions on the numerical results}
To the best of our knowledge, the FISTA algorithm \cite{Beck-Teboulle-2009} does not have a counterpart for strongly convex problem, but we still conducted some numerical experiments using FISTA for solving the above problems. We found that FISTA and APG are comparable, but they are both worse than GeoPG for more ill-conditioned problems. Moreover, from the results in this section, we can see that when the problem is well-posed such as $\alpha = 0.01$, APG is usually faster than GeoPG in the CPU time, and when the problem is ill-posed such as $\alpha = 10^{-6}$, $10^{-8}$, $10^{-10}$, GeoPG is usually faster, but the iterate of GeoPG is less than APG in the most cases. So GeoPG is not always better than APG in the CPU time. But since ill-posed problems are more challenging to solve, we believe that these numerical results showed the potential of GeoPG. The reason why GeoPG is better than APG for ill-posed problem is still not clear at this moment, but we think that it might be related to the fact that APG is not monotone but GeoPG is, which can be seen from the figures in our paper. Furthermore, although GeoPG requires to find the root of a function $\bar{\phi}$ in each iteration, we found that a very good approximation of the root can be obtained by running the semi-smooth Newton method for 1-2 iterations on average. This explains why these steps of GeoPG do not bring much trouble in practice.
}

\subsection{Numerical results of L-GeoPG-B}

In this subsection, we tested GeoPG-B with limited memory described in Algorithm 5 on solving \eqref{LR} on Gisette dataset.  The results for different memory size $m$ are reported in Table \ref{tab: mem log gisette}. Note that $m=0$ corresponds to the original GeoPG-B without memory.

From Table \ref{tab: mem log gisette} we see that roughly speaking, L-GeoPG-B performs better for larger memory size, and in almost all cases, the performance of L-GeoPG-B with $m=100$ is the best among the reported results. This indicates that the limited-memory idea indeed helps improve the performance of GeoPG.

\begin{table}[H]\scriptsize
 \caption{GeoPG-B and APG-B for solving linear regression with elastic net regularization. $p=4000, n=2000$}\label{tab:lasso-1}

  \centering
    \begin{tabular}{|c|c|c|c|c|c|c|c|c|c|c|c|c|c|}
    \hline
          & \multicolumn{6}{c|}{APG-B}                    & \multicolumn{6}{c|}{GeoPG-B}                &  \bigstrut\\
    \hline
    $\alpha$ & iter  & cpu   & f-ev  & g-ev  & p-ev   & MVM   & iter & cpu   & f-ev  & g-ev  & p-ev   & MVM   & f-diff \bigstrut\\
    \hline
    \hline
    \multicolumn{14}{|c|}{$\mu=1.136e-02  $} \bigstrut\\
    \hline
    \hline
    $10^{-2}$ & 172   & 1.0   & 354   & 326   & 194   & 384   & 156   & 1.1   & 457   & 348   & 352   & 398   & 8.5e-14 \bigstrut\\
    \hline
    $10^{-4}$ & 538   & 2.8   & 1116  & 1020  & 611   & 1203  & 95    & 0.7   & 267   & 240   & 245   & 247   & 6.4e-14 \bigstrut\\
    \hline
    $10^{-6}$ & 905   & 4.9   & 1868  & 1715  & 1029  & 2030  & 94    & 0.7   & 260   & 249   & 254   & 247   & 5.0e-14 \bigstrut\\
    \hline
    $10^{-8}$ & 1040  & 5.4   & 2146  & 2003  & 1182  & 2332  & 95    & 0.7   & 263   & 258   & 263   & 247   & 1.4e-14 \bigstrut\\
    \hline
    $10^{-10}$ & 964   & 5.0   & 2002  & 1805  & 1095  & 2154  & 95    & 0.7   & 263   & 267   & 272   & 247   & 2.1e-14 \bigstrut\\
    \hline
    \hline
    \multicolumn{14}{|c|}{$\mu=1.136e-03  $} \bigstrut\\
    \hline
    \hline
    $10^{-2}$ & 175   & 0.9   & 356   & 332   & 197   & 392   & 168   & 1.2   & 493   & 384   & 388   & 432   & 1.3e-13 \bigstrut\\
    \hline
    $10^{-4}$ & 687   & 3.6   & 1414  & 1304  & 779   & 1539  & 145   & 1.0   & 411   & 392   & 397   & 377   & 1.5e-14 \bigstrut\\
    \hline
    $10^{-6}$ & 999   & 5.1   & 2086  & 1676  & 1134  & 2225  & 140   & 1.0   & 371   & 384   & 394   & 354   & 6.5e-14 \bigstrut\\
    \hline
    $10^{-8}$ & 1122  & 5.8   & 2348  & 1827  & 1275  & 2499  & 143   & 1.0   & 374   & 420   & 429   & 365   & 1.8e-15 \bigstrut\\
    \hline
    $10^{-10}$ & 1142  & 5.9   & 2388  & 1858  & 1298  & 2545  & 143   & 1.0   & 374   & 449   & 458   & 365   & 6.2e-15 \bigstrut\\
    \hline
    \hline
    \multicolumn{14}{|c|}{$\mu=1.136e-04  $} \bigstrut\\
    \hline
    \hline
    $10^{-2}$ & 168   & 0.9   & 346   & 314   & 189   & 374   & 113   & 0.8   & 328   & 252   & 256   & 296   & 1.4e-14 \bigstrut\\
    \hline
    $10^{-4}$ & 911   & 4.8   & 1836  & 1853  & 1035  & 2064  & 207   & 1.5   & 603   & 587   & 592   & 535   & 4.1e-14 \bigstrut\\
    \hline
    $10^{-6}$ & 2293  & 11.9  & 4744  & 3936  & 2605  & 5132  & 191   & 1.4   & 523   & 596   & 602   & 492   & 3.8e-14 \bigstrut\\
    \hline
    $10^{-8}$ & 3979  & 20.5  & 8266  & 5923  & 4526  & 8899  & 199   & 1.4   & 500   & 713   & 728   & 501   & 9.8e-14 \bigstrut\\
    \hline
    $10^{-10}$ & 4503  & 23.3  & 9364  & 6668  & 5123  & 10068 & 185   & 1.3   & 456   & 624   & 639   & 465   & 5.9e-14 \bigstrut\\
     \hline
    \end{tabular}
\end{table}

\begin{table}[H]\scriptsize
 \caption{GeoPG-B and APG-B for solving linear regression with elastic net regularization. $p=2000, n=2000$ }\label{tab:lasso-2}

  \centering
\begin{tabular}{|c|c|c|c|c|c|c|c|c|c|c|c|c|c|}
\hline
      & \multicolumn{6}{c|}{APG-B}                    & \multicolumn{6}{c|}{GeoPG-B}                &  \bigstrut\\
\hline
$\alpha$ & iter  & cpu   & f-ev  & g-ev  & p-ev   & MVM   & iter & cpu   & f-ev  & g-ev  & p-ev   & MVM   & f-diff \bigstrut\\
\hline
\hline
\multicolumn{14}{|c|}{$\mu=1.50e-02  $} \bigstrut\\
\hline
\hline
$10^{-2}$ & 244   & 0.7   & 498   & 475   & 276   & 548   & 304   & 1.3   & 889   & 690   & 694   & 774   & 3.4e-13 \bigstrut\\
\hline
$10^{-4}$ & 1800  & 4.8   & 3690  & 3582  & 2046  & 4048  & 545   & 2.4   & 1569  & 1298  & 1308  & 1378  & 1.3e-12 \bigstrut\\
\hline
$10^{-6}$ & 9706  & 26.0  & 19722 & 20445 & 11040 & 21926 & 557   & 2.3   & 1598  & 1328  & 1339  & 1415  & 2.8e-12 \bigstrut\\
\hline
$10^{-8}$ & 20056 & 53.7  & 40528 & 43361 & 22817 & 45427 & 561   & 2.3   & 1614  & 1332  & 1344  & 1416  & 2.4e-12 \bigstrut\\
\hline
$10^{-10}$ & 20473 & 53.9  & 41426 & 44159 & 23298 & 46357 & 565   & 2.3   & 1626  & 1373  & 1385  & 1436  & 2.4e-12 \bigstrut\\
\hline
\hline
\multicolumn{14}{|c|}{$\mu=1.50e-03  $} \bigstrut\\
\hline
\hline
$10^{-2}$ & 241   & 0.6   & 496   & 463   & 273   & 540   & 280   & 1.2   & 813   & 634   & 638   & 716   & 1.4e-14 \bigstrut\\
\hline
$10^{-4}$ & 1926  & 5.1   & 3968  & 3708  & 2188  & 4319  & 1218  & 5.0   & 3560  & 2875  & 2892  & 3073  & 2.0e-11 \bigstrut\\
\hline
$10^{-6}$ & 12502 & 32.7  & 25658 & 24681 & 14222 & 28118 & 1297  & 5.3   & 3718  & 3065  & 3097  & 3262  & 1.1e-11 \bigstrut\\
\hline
$10^{-8}$ & 47139 & 124.3  & 95560 & 100584 & 53646 & 106652 & 1289  & 5.3   & 3686  & 3043  & 3074  & 3245  & 2.1e-11 \bigstrut\\
\hline
$10^{-10}$ & 72186 & 194.3  & 145934 & 156713 & 82157 & 163534 & 1297  & 5.2   & 3717  & 3098  & 3132  & 3262  & 2.5e-11 \bigstrut\\
\hline
\hline
\multicolumn{14}{|c|}{$\mu=1.50e-04 $} \bigstrut\\
\hline
\hline
$10^{-2}$ & 239   & 0.6   & 488   & 460   & 270   & 536   & 225   & 0.9   & 648   & 510   & 514   & 584   & 3.3e-13 \bigstrut\\
\hline
$10^{-4}$ & 1985  & 5.2   & 4048  & 3860  & 2257  & 4476  & 1713  & 6.9   & 5041  & 4040  & 4058  & 4322  & 7.0e-11 \bigstrut\\
\hline
$10^{-6}$ & 13824 & 35.7  & 28534 & 25354 & 15726 & 31010 & 2527  & 10.2  & 7225  & 6019  & 6082  & 6345  & 2.5e-11 \bigstrut\\
\hline
$10^{-8}$ & 56339 & 146.2  & 116280 & 106460 & 64105 & 126410 & 2594  & 10.6  & 7288  & 6095  & 6182  & 6491  & 3.6e-11 \bigstrut\\
\hline
$10^{-10}$ & $-$   & $-$   & $-$   & $-$   & $-$   & $-$   & 2573  & 10.4  & 7217  & 6075  & 6163  & 6446  & $-$ \bigstrut\\
\hline
\end{tabular}%

\end{table}

\begin{table}[H]\scriptsize
\caption{GeoPG-B and APG-B for solving linear regression with elastic net regularization. $p=2000, n=4000$} \label{tab:lasso-3}

  \centering

\begin{tabular}{|c|c|c|c|c|c|c|c|c|c|c|c|c|c|}
\hline
      & \multicolumn{6}{c|}{APG-B}                    & \multicolumn{6}{c|}{GeoPG-B}                &  \bigstrut\\
\hline
$\alpha$ & iter  & cpu   & f-ev  & g-ev  & p-ev   & MVM   & iter & cpu   & f-ev  & g-ev  & p-ev   & MVM   & f-diff \bigstrut\\
\hline
\hline
\multicolumn{14}{|c|}{$\mu= 1.82e-02 $} \bigstrut\\
\hline
\hline
$10^{-2}$ & 327   & 1.9   & 660   & 680   & 371   & 740   & 387   & 2.8   & 1117  & 936   & 946   & 980   & 2.0e-13 \bigstrut\\
\hline
$10^{-4}$ & 2263  & 12.8  & 4620  & 4445  & 2571  & 5096  & 2454  & 17.9  & 6858  & 6181  & 6225  & 6168  & 4.3e-11 \bigstrut\\
\hline
$10^{-6}$ & 12579 & 67.5  & 25566 & 26229 & 14312 & 28421 & 4478  & 32.7  & 12494 & 11180 & 11216 & 11300 & 1.8e-11 \bigstrut\\
\hline
$10^{-8}$ & 55577 & 299.3  & 112140 & 121939 & 63268 & 126044 & 4595  & 33.7  & 12814 & 11754 & 11795 & 11609 & 1.4e-10 \bigstrut\\
\hline
$10^{-10}$ & $-$   & $-$   & $-$   & $-$   & $-$   & $-$   & 4645  & 34.6  & 13204 & 12088 & 12129 & 11729 & $-$ \bigstrut\\
\hline
\hline
\multicolumn{14}{|c|}{$\mu= 1.82e-03$} \bigstrut\\
\hline
\hline
$10^{-2}$ & 306   & 1.7   & 622   & 621   & 346   & 688   & 279   & 2.1   & 813   & 677   & 684   & 713   & 6.4e-13 \bigstrut\\
\hline
$10^{-4}$ & 2355  & 12.7  & 4820  & 4534  & 2675  & 5296  & 2634  & 19.3  & 7482  & 6774  & 6846  & 6596  & 3.9e-13 \bigstrut\\
\hline
$10^{-6}$ & 14827 & 79.8  & 30328 & 28671 & 16862 & 33388 & 12756 & 94.1  & 36510 & 32580 & 32735 & 32121 & 2.2e-10 \bigstrut\\
\hline
$10^{-8}$ & 56286 & 305.7  & 114576 & 115199 & 64050 & 127099 & 11665 & 88.0  & 32397 & 32580 & 31987 & 29352 & 6.1e-11 \bigstrut\\
\hline
$10^{-10}$ & $-$   & $-$   & $-$   & $-$   & $-$   & $-$   & 13830 & 102.4  & 38547 & 37931 & 38088 & 34885 & $-$ \bigstrut\\
\hline
\hline
\multicolumn{14}{|c|}{$\mu= 1.82e-04 $} \bigstrut\\
\hline
\hline
$10^{-2}$ & 283   & 1.5   & 576   & 560   & 320   & 636   & 219   & 1.6   & 643   & 523   & 528   & 561   & 4.7e-13 \bigstrut\\
\hline
$10^{-4}$ & 2420  & 13.2  & 4864  & 5242  & 2749  & 5487  & 2339  & 17.2  & 6818  & 6467  & 6509  & 5882  & 5.8e-11 \bigstrut\\
\hline
$10^{-6}$ & 16882 & 91.4  & 34412 & 31337 & 19186 & 38049 & 14803 & 109.3  & 41943 & 44052 & 44384 & 37152 & 4.9e-10 \bigstrut\\
\hline
$10^{-8}$ & 79693 & 430.5  & 163098 & 146951 & 90639 & 179423 & 41331 & 305.8  & 116983 & 113344 & 113952 & 104206 & 1.6e-10 \bigstrut\\
\hline
$10^{-10}$ & $-$   & $-$   & $-$   & $-$   & $-$   & $-$   & 47501 & 350.2  & 129513 & 151332 & 152224 & 119660 & $-$ \bigstrut\\
\hline
\end{tabular}%

\end{table}

\begin{table}[H]\scriptsize
  \caption{GeoPG-B and APG-B for solving linear regression with elastic net regularization on dataset a9a}\label{tab: leq a9a}

  \centering
\begin{tabular}{|c|c|c|c|c|c|c|c|c|c|c|c|c|c|}
\hline
      & \multicolumn{6}{c|}{APG-B}                    & \multicolumn{6}{c|}{GeoPG-B}                &  \bigstrut\\
\hline
$\alpha$ & iter  & cpu   & f-ev  & g-ev  & p-ev   & MVM   & iter & cpu   & f-ev  & g-ev  & p-ev   & MVM   & f-diff \bigstrut\\
\hline
\hline
\multicolumn{14}{|c|}{$\lambda=1e-03 $} \bigstrut\\
\hline
\hline
$10^{-2}$ & 266   & 0.3   & 540   & 530   & 301   & 599   & 260   & 0.6   & 769   & 602   & 608   & 662   & 1.3e-14 \bigstrut\\
\hline
$10^{-4}$ & 1758  & 1.7   & 3562  & 3683  & 1998  & 3974  & 463   & 1.1   & 1374  & 1138  & 1144  & 1196  & 1.2e-14 \bigstrut\\
\hline
$10^{-6}$ & 10790 & 10.4  & 21654 & 23858 & 12277 & 24518 & 410   & 0.9   & 1216  & 964   & 970   & 1058  & 1.5e-13 \bigstrut\\
\hline
$10^{-8}$ & 23279 & 22.2  & 46646 & 52163 & 26493 & 52943 & 412   & 0.9   & 1222  & 976   & 982   & 1060  & 1.9e-13 \bigstrut\\
\hline
$10^{-10}$ & 26057 & 24.9  & 52236 & 58464 & 29660 & 59260 & 431   & 0.9   & 1279  & 1063  & 1069  & 1104  & 2.2e-13 \bigstrut\\
\hline
\hline
\multicolumn{14}{|c|}{$\lambda=1e-04 $} \bigstrut\\
\hline
\hline
$10^{-2}$ & 267   & 0.3   & 544   & 526   & 302   & 600   & 249   & 0.5   & 734   & 571   & 577   & 642   & 6.7e-16 \bigstrut\\
\hline
$10^{-4}$ & 1948  & 1.9   & 3934  & 4100  & 2214  & 4410  & 1587  & 3.4   & 4747  & 3946  & 3951  & 4025  & 2.9e-12 \bigstrut\\
\hline
$10^{-6}$ & 14954 & 14.3  & 30012 & 33215 & 17018 & 33985 & 4801  & 10.4  & 14388 & 11381 & 11386 & 12223 & 1.4e-12 \bigstrut\\
\hline
$10^{-8}$ & 63920 & 60.9  & 127954 & 144494 & 72741 & 145426 & 910   & 2.0   & 2715  & 2629  & 2634  & 2347  & 3.7e-12 \bigstrut\\
\hline
$10^{-10}$ & 94861 & 90.6  & 189814 & 214931 & 107970 & 215895 & 910   & 2.0   & 2715  & 2441  & 2446  & 2333  & 7.0e-13 \bigstrut\\
\hline
\hline
\multicolumn{14}{|c|}{$\lambda=1e-05 $} \bigstrut\\
\hline
\hline
$10^{-2}$ & 258   & 0.3   & 518   & 507   & 292   & 584   & 235   & 0.5   & 692   & 596   & 602   & 604   & 1.2e-14 \bigstrut\\
\hline
$10^{-4}$ & 2035  & 1.9   & 4088  & 4319  & 2315  & 4622  & 1701  & 3.7   & 5090  & 4267  & 4273  & 4312  & 3.7e-12 \bigstrut\\
\hline
$10^{-6}$ & 16353 & 15.6  & 32768 & 36396 & 18609 & 37188 & 5773  & 12.5  & 17306 & 14961 & 14967 & 14808 & 4.5e-13 \bigstrut\\
\hline
$10^{-8}$ & 85246 & 81.4  & 170570 & 193007 & 97062 & 194086 & 2109  & 4.6   & 6314  & 6403  & 6409  & 5382  & 2.5e-11 \bigstrut\\
\hline
$10^{-10}$ & $-$   & $-$   & $-$   & $-$   & $-$   & $-$   & 2318  & 5.0   & 6941  & 6709  & 6715  & 5896  & $-$ \bigstrut\\
\hline
\end{tabular}%

\end{table}

\begin{table}[htbp]\scriptsize
  \caption{GeoPG-B and APG-B for solving linear regression with elastic net regularization on dataset rcv1}\label{tab: leq rcv1}

  \centering
\begin{tabular}{|c|c|c|c|c|c|c|c|c|c|c|c|c|c|}
\hline
      & \multicolumn{6}{c|}{APG-B}                    & \multicolumn{6}{c|}{GeoProx-B}                &  \bigstrut\\
\hline
$\alpha$ & iter  & cpu   & f-ev  & g-ev  & p-ev   & MVM   & f-diff & cpu   & f-ev  & g-ev  & p-ev   & MVM   & f-diff \bigstrut\\
\hline
\hline
\multicolumn{14}{|c|}{$\lambda=1e-03 $} \bigstrut\\
\hline
\hline
$10^{-2}$ & 18    & 0.1   & 34    & 34    & 20    & 42    & 14    & 0.2   & 39    & 31    & 32    & 43    & 5.5e-14 \bigstrut\\
\hline
$10^{-4}$ & 74    & 0.3   & 148   & 141   & 82    & 165   & 95    & 0.7   & 273   & 231   & 232   & 245   & 7.8e-13 \bigstrut\\
\hline
$10^{-6}$ & 329   & 1.5   & 678   & 617   & 372   & 735   & 103   & 0.8   & 296   & 265   & 268   & 269   & 6.6e-13 \bigstrut\\
\hline
$10^{-8}$ & 908   & 4.2   & 1872  & 1721  & 1033  & 2039  & 133   & 1.0   & 380   & 344   & 345   & 341   & 7.0e-13 \bigstrut\\
\hline
$10^{-10}$ & 1277  & 5.9   & 2630  & 2482  & 1454  & 2871  & 116   & 0.9   & 332   & 331   & 332   & 301   & 1.1e-12 \bigstrut\\
\hline
\hline
\multicolumn{14}{|c|}{$\lambda=1e-04 $} \bigstrut\\
\hline
\hline
$10^{-2}$ & 17    & 0.1   & 32    & 31    & 19    & 40    & 17    & 0.1   & 48    & 34    & 35    & 49    & 1.6e-13 \bigstrut\\
\hline
$10^{-4}$ & 109   & 0.5   & 226   & 195   & 123   & 243   & 109   & 0.5   & 226   & 195   & 123   & 243   & 1.7e-12 \bigstrut\\
\hline
$10^{-6}$ & 723   & 3.2   & 1482  & 1401  & 821   & 1625  & 251   & 1.9   & 743   & 625   & 633   & 634   & 1.3e-11 \bigstrut\\
\hline
$10^{-8}$ & 3087  & 13.9  & 6276  & 6426  & 3513  & 6976  & 247   & 2.1   & 723   & 645   & 653   & 626   & 1.4e-11 \bigstrut\\
\hline
$10^{-10}$ & 5266  & 23.7  & 10638 & 11244 & 5991  & 11930 & 244   & 1.9   & 711   & 672   & 678   & 624   & 6.0e-12 \bigstrut\\
\hline
\hline
\multicolumn{14}{|c|}{$\lambda=1e-05 $} \bigstrut\\
\hline
\hline
$10^{-2}$ & 16    & 0.1   & 30    & 28    & 18    & 38    & 15    & 0.1   & 42    & 32    & 33    & 45    & 3.1e-13 \bigstrut\\
\hline
$10^{-4}$ & 118   & 0.5   & 240   & 220   & 134   & 267   & 125   & 1.0   & 359   & 289   & 294   & 321   & 1.0e-10 \bigstrut\\
\hline
$10^{-6}$ & 859   & 3.9   & 1750  & 1595  & 978   & 1941  & 833   & 6.8   & 2470  & 2186  & 2199  & 2105  & 5.7e-10 \bigstrut\\
\hline
$10^{-8}$ & 5902  & 26.5  & 11918 & 11933 & 6716  & 13376 & 1179  & 9.6   & 3509  & 3336  & 3348  & 2998  & 1.4e-09 \bigstrut\\
\hline
$10^{-10}$ & 33127 & 150.7  & 66438 & 72792 & 37722 & 75353 & 1180  & 9.7   & 3508  & 3540  & 3555  & 2995  & 7.2e-10 \bigstrut\\
\hline
\end{tabular}%

\end{table}

\begin{table}[H]\scriptsize
 \caption{GeoPG-B and APG-B for solving linear regression with elastic net regularization on data set Gisette. Note that neither APG-B nor GeoPG-B converges in $10^5$ iterations when $\mu=1e-05 $ and $\alpha= 10^{-6},10^{-8},10^{-10}$.}\label{tab: leq gisette}

  \centering
\begin{tabular}{|c|c|c|c|c|c|c|c|c|c|c|c|c|c|}
\hline
      & \multicolumn{6}{c|}{APG-B}                    & \multicolumn{6}{c|}{GeoPG-B}                &  \bigstrut\\
\hline
$\alpha$ & iter  & cpu   & f-ev  & g-ev  & p-ev   & MVM   & iter & cpu   & f-ev  & g-ev  & p-ev   & MVM   & f-diff \bigstrut\\
\hline
\hline
\multicolumn{14}{|c|}{$\mu=1e-03 $} \bigstrut\\
\hline
\hline
$10^{-2}$ & 4026  & 198.1  & 8144  & 7729  & 4583  & 9121  & 4253  & 239.3  & 12593 & 10474 & 10506 & 10758 & 4.8e-14 \bigstrut\\
\hline
$10^{-4}$ & 30537 & 1504.2  & 61478 & 61380 & 34786 & 69371 & 6030  & 342.4  & 17939 & 17977 & 18006 & 15411 & 1.6e-13 \bigstrut\\
\hline
$10^{-6}$ & $-$   & $-$   & $-$   & $-$   & $-$   & $-$   & 5197  & 294.0  & 15419 & 16126 & 16159 & 13241 & $-$ \bigstrut\\
\hline
$10^{-8}$ & $-$   & $-$   & $-$   & $-$   & $-$   & $-$   & 5692  & 322.8  & 16950 & 18851 & 18881 & 14506 & $-$ \bigstrut\\
\hline
$10^{-10}$ & $-$   & $-$   & $-$   & $-$   & $-$   & $-$   & 6150  & 353.5  & 18295 & 23420 & 23450 & 15714 & $-$ \bigstrut\\
\hline
\hline
\multicolumn{14}{|c|}{$\mu=1e-04 $} \bigstrut\\
\hline
\hline
$10^{-2}$ & 6084  & 299.5  & 12288 & 12211 & 6930  & 13801 & 5406  & 304.3  & 16046 & 13623 & 13658 & 13675 & 1.1e-13 \bigstrut\\
\hline
$10^{-4}$ & 49467 & 2434.4  & 99880 & 100633 & 56333 & 112194 & 36606 & 2046.7  & 105023 & 112545 & 113414 & 91853 & 1.6e-13 \bigstrut\\
\hline
$10^{-6}$ & $-$   & $-$   & $-$   & $-$   & $-$   & $-$   & 20821 & 1179.7  & 62243 & 65886 & 65919 & 53105 & $-$ \bigstrut\\
\hline
$10^{-8}$ & $-$   & $-$   & $-$   & $-$   & $-$   & $-$   & 21575 & 1224.1  & 64488 & 71718 & 71753 & 54979 & $-$ \bigstrut\\
\hline
$10^{-10}$ & $-$   & $-$   & $-$   & $-$   & $-$   & $-$   & 20328 & 1164.9  & 60730 & 76896 & 76942 & 51908 & $-$ \bigstrut\\
\hline
\hline
\multicolumn{14}{|c|}{$\mu=1e-05 $} \bigstrut\\
\hline
\hline
$10^{-2}$ & 6570  & 323.9  & 13304 & 13289 & 7483  & 14885 & 4803  & 270.8  & 14228 & 11515 & 11547 & 12164 & 2.7e-13 \bigstrut\\
\hline
$10^{-4}$ & 56562 & 2791.0  & 114250 & 115944 & 64396 & 128230 & 38001 & 2153.4  & 113603 & 100036 & 100105 & 96725 & 5.6e-12 \bigstrut\\
\hline
\end{tabular}%

\end{table}

\begin{table}[htbp]\scriptsize
  \caption{GeoPG-B and APG-B for solving logistic regression with elastic net regularization. $p=6000, n=3000$} \label{tab:LR-Syn-1}

  \centering
\begin{tabular}{|c|c|c|c|c|c|c|c|c|c|c|c|c|c|}
\hline
      & \multicolumn{6}{c|}{APG-B}                    & \multicolumn{6}{c|}{GeoPG-B}                &  \bigstrut\\
\hline
$\alpha$ & iter  & cpu   & f-ev  & g-ev  & p-ev   & MVM   & iter & cpu   & f-ev  & g-ev  & p-ev   & MVM   & f-diff \bigstrut\\
\hline
\hline
\multicolumn{14}{|c|}{$\mu=1.00e-03  $} \bigstrut\\
\hline
\hline
$10^{-2}$ & 55    & 0.9   & 112   & 96    & 60    & 158   & 46    & 1.3   & 125   & 145   & 146   & 207   & 1.1e-13 \bigstrut\\
\hline
$10^{-4}$ & 256   & 4.3   & 536   & 470   & 289   & 761   & 55    & 1.7   & 144   & 194   & 194   & 269   & 5.6e-13 \bigstrut\\
\hline
$10^{-6}$ & 509   & 8.7   & 1048  & 972   & 577   & 1551  & 61    & 2.0   & 164   & 218   & 220   & 300   & 1.3e-12 \bigstrut\\
\hline
$10^{-8}$ & 573   & 9.5   & 1188  & 1086  & 649   & 1737  & 60    & 1.9   & 161   & 223   & 225   & 305   & 1.4e-12 \bigstrut\\
\hline
$10^{-10}$ & 585   & 9.6   & 1208  & 1112  & 663   & 1777  & 59    & 2.1   & 158   & 231   & 233   & 313   & 1.4e-12 \bigstrut\\
\hline
\hline
\multicolumn{14}{|c|}{$\mu=1.00e-04 $} \bigstrut\\
\hline
\hline
$10^{-2}$ & 51    & 0.7   & 104   & 80    & 55    & 137   & 51    & 1.3   & 141   & 167   & 164   & 236   & 2.5e-13 \bigstrut\\
\hline
$10^{-4}$ & 203   & 3.0   & 422   & 336   & 226   & 564   & 118   & 3.2   & 319   & 405   & 396   & 555   & 1.3e-11 \bigstrut\\
\hline
$10^{-6}$ & 954   & 14.7  & 1994  & 1662  & 1080  & 2744  & 126   & 3.8   & 335   & 452   & 450   & 614   & 3.1e-11 \bigstrut\\
\hline
$10^{-8}$ & 1814  & 28.7  & 3780  & 3311  & 2056  & 5369  & 125   & 3.7   & 336   & 454   & 454   & 614   & 2.6e-11 \bigstrut\\
\hline
$10^{-10}$ & 2135  & 33.9  & 4444  & 3952  & 2421  & 6375  & 125   & 4.0   & 336   & 475   & 475   & 635   & 3.2e-11 \bigstrut\\
\hline
\hline
\multicolumn{14}{|c|}{$\mu=1.00e-05  $} \bigstrut\\
\hline
\hline
$10^{-2}$ & 52    & 0.8   & 102   & 88    & 57    & 147   & 40    & 1.0   & 107   & 129   & 128   & 184   & 2.3e-13 \bigstrut\\
\hline
$10^{-4}$ & 141   & 1.9   & 288   & 208   & 154   & 364   & 97    & 2.4   & 257   & 316   & 309   & 438   & 3.3e-11 \bigstrut\\
\hline
$10^{-6}$ & 576   & 7.8   & 1246  & 804   & 646   & 1452  & 139   & 4.0   & 350   & 496   & 488   & 669   & 5.6e-11 \bigstrut\\
\hline
$10^{-8}$ & 2797  & 38.4  & 6070  & 4014  & 3166  & 7182  & 148   & 4.4   & 372   & 538   & 535   & 723   & 4.0e-10 \bigstrut\\
\hline
$10^{-10}$ & 4549  & 63.3  & 9862  & 6703  & 5151  & 11856 & 153   & 4.9   & 392   & 585   & 583   & 776   & 6.2e-10 \bigstrut\\
\hline
\end{tabular}%

\end{table}

\begin{table}[htbp]\scriptsize
    \caption{GeoPG-B and APG-B for solving logistic regression with elastic net regularization. $p=3000, n=6000$ } \label{tab:LR-Syn-2}

  \centering
\begin{tabular}{|c|c|c|c|c|c|c|c|c|c|c|c|c|c|}
\hline
      & \multicolumn{6}{c|}{APG-B}                    & \multicolumn{6}{c|}{GeoPG-B}                &  \bigstrut\\
\hline
$\alpha$ & iter  & cpu   & f-ev  & g-ev  & p-ev   & MVM   & iter & cpu   & f-ev  & g-ev  & p-ev   & MVM   & f-diff \bigstrut\\
\hline
\hline
\multicolumn{14}{|c|}{$\mu=1.00e-03  $} \bigstrut\\
\hline
\hline
$10^{-2}$ & 58    & 0.9   & 114   & 107   & 63    & 172   & 60    & 1.6   & 169   & 200   & 196   & 279   & 5.1e-14 \bigstrut\\
\hline
$10^{-4}$ & 253   & 4.1   & 516   & 466   & 284   & 752   & 110   & 3.5   & 292   & 420   & 412   & 562   & 1.9e-12 \bigstrut\\
\hline
$10^{-6}$ & 893   & 15.1  & 1824  & 1757  & 1012  & 2771  & 115   & 4.3   & 305   & 467   & 463   & 615   & 4.1e-12 \bigstrut\\
\hline
$10^{-8}$ & 1265  & 21.9  & 2584  & 2543  & 1435  & 3980  & 114   & 4.4   & 302   & 504   & 501   & 649   & 4.9e-12 \bigstrut\\
\hline
$10^{-10}$ & 1333  & 22.6  & 2712  & 2691  & 1513  & 4206  & 114   & 4.8   & 302   & 543   & 540   & 688   & 5.0e-12 \bigstrut\\
\hline
\hline
\multicolumn{14}{|c|}{$\mu=1.00e-04 $} \bigstrut\\
\hline
\hline
$10^{-2}$ & 56    & 0.8   & 112   & 89    & 60    & 151   & 42    & 1.1   & 116   & 133   & 132   & 188   & 1.4e-13 \bigstrut\\
\hline
$10^{-4}$ & 159   & 2.2   & 328   & 237   & 174   & 413   & 128   & 3.7   & 340   & 455   & 447   & 616   & 1.7e-11 \bigstrut\\
\hline
$10^{-6}$ & 750   & 11.3  & 1560  & 1238  & 845   & 2085  & 157   & 5.2   & 392   & 621   & 614   & 817   & 5.3e-11 \bigstrut\\
\hline
$10^{-8}$ & 1927  & 30.3  & 4012  & 3447  & 2182  & 5631  & 158   & 5.8   & 410   & 679   & 674   & 877   & 8.6e-11 \bigstrut\\
\hline
$10^{-10}$ & 2364  & 37.5  & 4934  & 4290  & 2677  & 6969  & 164   & 6.6   & 427   & 760   & 753   & 965   & 1.5e-10 \bigstrut\\
\hline
\hline
\multicolumn{14}{|c|}{$\mu=1.00e-05  $} \bigstrut\\
\hline
\hline
$10^{-2}$ & 54    & 0.8   & 108   & 85    & 58    & 145   & 42    & 1.1   & 110   & 136   & 134   & 191   & 2.9e-13 \bigstrut\\
\hline
$10^{-4}$ & 118   & 1.6   & 236   & 177   & 126   & 305   & 81    & 2.1   & 207   & 266   & 263   & 365   & 1.4e-11 \bigstrut\\
\hline
$10^{-6}$ & 493   & 6.4   & 1062  & 636   & 551   & 1189  & 153   & 4.9   & 365   & 588   & 580   & 776   & 2.9e-10 \bigstrut\\
\hline
$10^{-8}$ & 3492  & 45.0  & 7742  & 4365  & 3949  & 8316  & 163   & 5.8   & 379   & 686   & 677   & 886   & 8.3e-10 \bigstrut\\
\hline
$10^{-10}$ & 7655  & 98.4  & 17058 & 9498  & 8666  & 18166 & 169   & 6.8   & 403   & 782   & 775   & 990   & 1.7e-09 \bigstrut\\
\hline
\end{tabular}%

\end{table}%

\begin{table}[htbp]\scriptsize
\caption{GeoPG-B and APG-B for solving logistic regression with elastic net regularization. $p=3000, n=3000$ } \label{tab:LR-Syn-3}

  \centering
\begin{tabular}{|c|c|c|c|c|c|c|c|c|c|c|c|c|c|}
\hline
      & \multicolumn{6}{c|}{APG-B}                    & \multicolumn{6}{c|}{GeoPG-B}                &  \bigstrut\\
\hline
$\alpha$ & iter  & cpu   & f-ev  & g-ev  & p-ev   & MVM   & iter & cpu   & f-ev  & g-ev  & p-ev   & MVM   & f-diff \bigstrut\\
\hline
\hline
\multicolumn{14}{c|}{$\mu=1.00e-03  $} \bigstrut\\
\hline
\hline
$10^{-2}$ & 55    & 0.5   & 110   & 99    & 60    & 161   & 53    & 0.8   & 144   & 172   & 171   & 243   & 2.7e-13 \bigstrut\\
\hline
$10^{-4}$ & 278   & 2.4   & 566   & 512   & 312   & 826   & 90    & 1.4   & 237   & 325   & 322   & 442   & 2.7e-12 \bigstrut\\
\hline
$10^{-6}$ & 845   & 7.1   & 1732  & 1637  & 957   & 2596  & 89    & 1.5   & 234   & 336   & 334   & 452   & 2.6e-12 \bigstrut\\
\hline
$10^{-8}$ & 1158  & 9.7   & 2378  & 2283  & 1314  & 3599  & 89    & 1.6   & 234   & 361   & 359   & 477   & 2.6e-12 \bigstrut\\
\hline
$10^{-10}$ & 1186  & 9.9   & 2444  & 2340  & 1345  & 3687  & 88    & 1.7   & 231   & 377   & 375   & 492   & 2.8e-12 \bigstrut\\
\hline
\hline
\multicolumn{14}{c|}{$\mu=1.00e-04 $} \bigstrut\\
\hline
\hline
$10^{-2}$ & 55    & 0.4   & 108   & 89    & 60    & 151   & 53    & 0.7   & 144   & 172   & 169   & 242   & 3.5e-13 \bigstrut\\
\hline
$10^{-4}$ & 172   & 1.3   & 352   & 273   & 191   & 466   & 122   & 1.8   & 327   & 424   & 415   & 579   & 3.2e-11 \bigstrut\\
\hline
$10^{-6}$ & 868   & 6.6   & 1834  & 1455  & 980   & 2437  & 145   & 2.3   & 374   & 529   & 523   & 714   & 6.8e-11 \bigstrut\\
\hline
$10^{-8}$ & 1985  & 16.0  & 4168  & 3527  & 2248  & 5777  & 144   & 2.5   & 372   & 565   & 563   & 747   & 5.9e-11 \bigstrut\\
\hline
$10^{-10}$ & 2475  & 19.9  & 5160  & 4545  & 2807  & 7354  & 143   & 2.7   & 365   & 607   & 605   & 787   & 7.4e-11 \bigstrut\\
\hline
\hline
\multicolumn{14}{c|}{$\mu=1.00e-05  $} \bigstrut\\
\hline
\hline
$10^{-2}$ & 55    & 0.4   & 108   & 91    & 59    & 152   & 48    & 0.7   & 129   & 158   & 155   & 224   & 6.7e-13 \bigstrut\\
\hline
$10^{-4}$ & 126   & 0.9   & 256   & 185   & 137   & 324   & 126   & 0.9   & 256   & 185   & 137   & 324   & 2.0e-12 \bigstrut\\
\hline
$10^{-6}$ & 515   & 3.4   & 1108  & 680   & 576   & 1258  & 146   & 2.2   & 344   & 524   & 517   & 705   & 4.6e-10 \bigstrut\\
\hline
$10^{-8}$ & 3196  & 21.0  & 7054  & 4118  & 3615  & 7735  & 154   & 2.5   & 372   & 587   & 586   & 778   & 8.9e-10 \bigstrut\\
\hline
$10^{-10}$ & 6434  & 42.7  & 14228 & 8384  & 7284  & 15670 & 152   & 2.8   & 370   & 630   & 629   & 820   & 5.4e-10 \bigstrut\\
\hline
\end{tabular}%

\end{table}

\begin{table}[H]\scriptsize
\caption{GeoPG-B and APG-B for solving logistic regression with elastic net on dataset a9a }\label{tab: log a9a}

  \centering
\begin{tabular}{|c|c|c|c|c|c|c|c|c|c|c|c|c|c|}
\hline
      & \multicolumn{6}{c|}{APG-B}                    & \multicolumn{6}{c|}{GeoPG-B}                &  \bigstrut\\
\hline
$\alpha$ & iter  & cpu   & f-ev  & g-ev  & p-ev   & MVM   & iter & cpu   & f-ev  & g-ev  & p-ev   & MVM   & f-diff \bigstrut\\
\hline
\hline
\multicolumn{14}{|c|}{$\mu=1.00e-03  $} \bigstrut\\
\hline
\hline
$10^{-2}$ & 99    & 0.3   & 196   & 189   & 111   & 302   & 96    & 0.5   & 280   & 325   & 318   & 450   & 2.9e-15 \bigstrut\\
\hline
$10^{-4}$ & 676   & 1.8   & 1380  & 1317  & 766   & 2085  & 676   & 1.8   & 1380  & 1317  & 766   & 2085  & 1.7e-14 \bigstrut\\
\hline
$10^{-6}$ & 2696  & 6.8   & 5484  & 5466  & 3065  & 8533  & 187   & 1.0   & 540   & 663   & 683   & 885   & 2.6e-14 \bigstrut\\
\hline
$10^{-8}$ & 3911  & 9.8   & 7934  & 8114  & 4445  & 12561 & 188   & 1.0   & 545   & 654   & 678   & 876   & 2.0e-14 \bigstrut\\
\hline
$10^{-10}$ & 4324  & 10.9  & 8770  & 9013  & 4917  & 13932 & 200   & 1.1   & 581   & 758   & 783   & 991   & 5.9e-14 \bigstrut\\
\hline
\hline
\multicolumn{14}{|c|}{$\mu=1.00e-04 $} \bigstrut\\
\hline
\hline
$10^{-2}$ & 96    & 0.2   & 194   & 174   & 106   & 282   & 96    & 0.2   & 194   & 174   & 106   & 282   & 1.1e-14 \bigstrut\\
\hline
$10^{-4}$ & 709   & 1.7   & 1440  & 1388  & 805   & 2195  & 756   & 3.8   & 2251  & 2669  & 2577  & 3615  & 8.2e-13 \bigstrut\\
\hline
$10^{-6}$ & 5195  & 13.6  & 10488 & 10973 & 5912  & 16887 & 2581  & 13.4  & 7725  & 8995  & 8770  & 12112 & 4.4e-11 \bigstrut\\
\hline
$10^{-8}$ & 25300 & 64.8  & 50772 & 56141 & 28793 & 84936 & 716   & 3.7   & 2130  & 2529  & 2583  & 3427  & 9.9e-10 \bigstrut\\
\hline
$10^{-10}$ & 42633 & 109.4  & 85446 & 95447 & 48519 & 143968 & 723   & 3.8   & 2151  & 2584  & 2640  & 3497  & 7.9e-11 \bigstrut\\
\hline
\hline
\multicolumn{14}{|c|}{$\mu=1.00e-05  $} \bigstrut\\
\hline
\hline
$10^{-2}$ & 106   & 0.3   & 210   & 199   & 119   & 320   & 72    & 0.4   & 207   & 258   & 255   & 347   & 1.4e-14 \bigstrut\\
\hline
$10^{-4}$ & 770   & 1.9   & 1550  & 1526  & 874   & 2402  & 685   & 3.5   & 2038  & 2448  & 2367  & 3301  & 3.7e-12 \bigstrut\\
\hline
$10^{-6}$ & 5842  & 14.7  & 11762 & 12434 & 6648  & 19084 & 3026  & 15.5  & 9061  & 11099 & 10715 & 14750 & 2.1e-11 \bigstrut\\
\hline
$10^{-8}$ & 46819 & 119.9  & 93782 & 104946 & 53311 & 158259 & 7784  & 38.8  & 23335 & 26969 & 26326 & 36558 & 1.9e-12 \bigstrut\\
\hline
$10^{-10}$ & $-$   & $-$   & $-$   & $-$   & $-$   & $-$   & 1488  & 8.2   & 4447  & 5674  & 5721  & 7567  & $-$ \bigstrut\\
\hline
\end{tabular}%

\end{table}

\begin{table}[H]\scriptsize

  \caption{GeoPG-B and APG-B for solving logistic regression with elastic net on dataset RCV1 }\label{tab: log rcv1}
  \centering
\begin{tabular}{|c|c|c|c|c|c|c|c|c|c|c|c|c|c|}
\hline
      & \multicolumn{6}{c|}{APG-B}                    & \multicolumn{6}{c|}{GeoPG-B}                &  \bigstrut\\
\hline
$\alpha$ & iter  & cpu   & f-ev  & g-ev  & p-ev   & MVM   & iter & cpu   & f-ev  & g-ev  & p-ev   & MVM   & f-diff \bigstrut\\
\hline
\hline
\multicolumn{14}{|c|}{$\mu=1e-03 $} \bigstrut\\
\hline
\hline
$10^{-2}$ & 15    & 0.1   & 28    & 26    & 17    & 45    & 7     & 0.1   & 21    & 22    & 23    & 36    & 5.0e-14 \bigstrut\\
\hline
$10^{-4}$ & 35    & 0.2   & 68    & 61    & 37    & 100   & 30    & 0.3   & 83    & 91    & 93    & 134   & 7.7e-13 \bigstrut\\
\hline
$10^{-6}$ & 112   & 0.7   & 224   & 213   & 125   & 340   & 43    & 0.5   & 120   & 136   & 140   & 193   & 1.3e-12 \bigstrut\\
\hline
$10^{-8}$ & 196   & 1.2   & 390   & 384   & 220   & 606   & 39    & 0.5   & 110   & 137   & 138   & 191   & 4.3e-12 \bigstrut\\
\hline
$10^{-10}$ & 230   & 1.4   & 466   & 444   & 259   & 705   & 39    & 0.5   & 110   & 149   & 150   & 203   & 1.3e-12 \bigstrut\\
\hline
\hline
\multicolumn{14}{|c|}{$\mu=1e-04 $} \bigstrut\\
\hline
\hline
$10^{-2}$ & 13    & 0.1   & 24    & 22    & 15    & 39    & 11    & 0.1   & 32    & 35    & 36    & 56    & 3.0e-13 \bigstrut\\
\hline
$10^{-4}$ & 40    & 0.3   & 80    & 75    & 44    & 121   & 42    & 0.5   & 122   & 136   & 135   & 193   & 3.8e-12 \bigstrut\\
\hline
$10^{-6}$ & 178   & 1.0   & 368   & 311   & 200   & 513   & 153   & 1.8   & 431   & 542   & 527   & 738   & 6.2e-11 \bigstrut\\
\hline
$10^{-8}$ & 1039  & 6.2   & 2122  & 1941  & 1179  & 3122  & 137   & 1.6   & 384   & 495   & 503   & 673   & 1.9e-11 \bigstrut\\
\hline
$10^{-10}$ & 1983  & 11.7  & 4080  & 3724  & 2251  & 5977  & 137   & 1.7   & 379   & 526   & 524   & 705   & 2.5e-12 \bigstrut\\
\hline
\hline
\multicolumn{14}{|c|}{$\mu=1e-05 $} \bigstrut\\
\hline
\hline
$10^{-2}$ & 13    & 0.1   & 24    & 22    & 15    & 39    & 9     & 0.1   & 26    & 28    & 29    & 45    & 2.3e-13 \bigstrut\\
\hline
$10^{-4}$ & 42    & 0.2   & 84    & 71    & 46    & 119   & 39    & 0.4   & 108   & 120   & 122   & 173   & 2.3e-11 \bigstrut\\
\hline
$10^{-6}$ & 164   & 0.9   & 338   & 266   & 182   & 450   & 208   & 2.5   & 592   & 747   & 729   & 1012  & 1.4e-09 \bigstrut\\
\hline
$10^{-8}$ & 1115  & 6.4   & 2274  & 2013  & 1266  & 3281  & 377   & 4.7   & 1073  & 1436  & 1410  & 1901  & 5.5e-09 \bigstrut\\
\hline
$10^{-10}$ & 5569  & 33.5  & 11314 & 11135 & 6334  & 17471 & 486   & 6.2   & 1399  & 1930  & 1897  & 2542  & 6.5e-10 \bigstrut\\
\hline
\end{tabular}%

\end{table}

\begin{table}[H]\scriptsize

  \caption{GeoPG-B and APG-B for solving logistic regression with elastic net on dataset Gisette}\label{tab: log gisette}
  \centering
\begin{tabular}{|c|c|c|c|c|c|c|c|c|c|c|c|c|c|}
\hline
      & \multicolumn{6}{c|}{APG-B}                    & \multicolumn{6}{c|}{GeoPG-B}                &  \bigstrut\\
\hline
$\alpha$ & iter  & cpu   & f-ev  & g-ev  & p-ev   & MVM   & iter & cpu   & f-ev  & g-ev  & p-ev   & MVM   & f-diff \bigstrut\\
\hline
\hline
\multicolumn{14}{|c|}{$\mu=1e-03 $} \bigstrut\\
\hline
\hline
$10^{-2}$ & 630   & 40.4  & 1267  & 1176  & 717   & 1895  & 819   & 82.5  & 2298  & 2867  & 2790  & 3903  & 2.8e-14 \bigstrut\\
\hline
$10^{-4}$ & 2445  & 156.0  & 4923  & 4511  & 2784  & 7297  & 2177  & 217.5  & 6197  & 7710  & 7477  & 10406 & 3.9e-13 \bigstrut\\
\hline
$10^{-6}$ & 13950 & 915.2  & 28209 & 28106 & 15889 & 43997 & 2013  & 230.9  & 5654  & 7676  & 7737  & 10200 & 2.0e-12 \bigstrut\\
\hline
$10^{-8}$ & 64288 & 4397.1  & 129271 & 140483 & 73191 & 213676 & 1793  & 214.9  & 5033  & 7146  & 7188  & 9371  & 4.4e-14 \bigstrut\\
\hline
$10^{-10}$ & $-$   & $-$   & $-$   & $-$   & $-$   & $-$   & 1808  & 227.1  & 5079  & 7532  & 7559  & 9783  & $-$ \bigstrut\\
\hline
\hline
\multicolumn{14}{|c|}{$\mu=1e-04 $} \bigstrut\\
\hline
\hline
$10^{-2}$ & 913   & 57.7  & 1845  & 1744  & 1041  & 2787  & 961   & 93.7  & 2740  & 3335  & 3237  & 4553  & 3.5e-13 \bigstrut\\
\hline
$10^{-4}$ & 1889  & 113.1  & 3811  & 3246  & 2150  & 5398  & 913   & 57.7  & 1845  & 1744  & 1041  & 2787  & 3.9e-12 \bigstrut\\
\hline
$10^{-6}$ & 10206 & 614.4  & 20687 & 17730 & 11632 & 29364 & 2243  & 258.2  & 6044  & 8768  & 8763  & 11486 & 3.0e-11 \bigstrut\\
\hline
$10^{-8}$ & 53272 & 3405.7  & 107397 & 103641 & 60702 & 164345 & 2226  & 276.7  & 6001  & 9318  & 9300  & 12002 & 2.8e-11 \bigstrut\\
\hline
$10^{-10}$ & $-$   & $-$   & $-$   & $-$   & $-$   & $-$   & 2203  & 296.2  & 5926  & 9809  & 9812  & 12488 & $-$ \bigstrut\\
\hline
\hline
\multicolumn{14}{|c|}{$\mu=1e-05 $} \bigstrut\\
\hline
\hline
$10^{-2}$ & 975   & 63.2  & 1981  & 1882  & 1110  & 2994  & 795   & 79.8  & 2242  & 2738  & 2662  & 3747  & 6.5e-13 \bigstrut\\
\hline
$10^{-4}$ & 1485  & 91.1  & 3019  & 2632  & 1687  & 4321  & 1381  & 141.4  & 3760  & 4943  & 4829  & 6686  & 6.8e-12 \bigstrut\\
\hline
$10^{-6}$ & 4642  & 265.8  & 9439  & 7240  & 5286  & 12528 & 2928  & 313.0  & 7964  & 10940 & 10698 & 14554 & 1.5e-11 \bigstrut\\
\hline
$10^{-8}$ & 29242 & 1681.8  & 59811 & 46411 & 33325 & 79738 & 2946  & 374.2  & 7789  & 12617 & 12543 & 16128 & 5.5e-10 \bigstrut\\
\hline
$10^{-10}$ & $-$   & $-$   & $-$   & $-$   & $-$   & $-$   & 2776  & 436.6  & 7359  & 13607 & 13563 & 16936 & $-$ \bigstrut\\
\hline
\end{tabular}%

\end{table}

\begin{table}[t]\scriptsize
 \caption{L-GeoPG-B for solving logistic regression with elastic net regularization on data set Gisette}\label{tab: mem log gisette}
  \centering
    \begin{tabular}{|c|c|c|c|c|c|c|c|c|c|c|c|c|}
    \hline
     & \multicolumn{2}{c|}{$m=0$} & \multicolumn{2}{c|}{$m=5$} & \multicolumn{2}{c|}{$m=10$} & \multicolumn{2}{c|}{$m=20$} & \multicolumn{2}{c|}{$m=50$} & \multicolumn{2}{c|}{$m=100$} \bigstrut\\
\cline{1-13}    $\alpha$      & iter  & cpu   & iter  & cpu   & iter  & cpu   & iter  & cpu   & iter  & cpu   & iter  & cpu \bigstrut\\
    \hline
    \hline
    \multicolumn{13}{|c|}{$\mu=1e-03 $} \bigstrut\\
    \hline
    \hline
    $10^{-2}$ &819   & 82.5 & 1310  & 164.1 & 1015  & 125.8 & 902   & 97.0 & 713   & 76.6 & 769   & 93.6 \bigstrut\\
    \hline
    $10^{-4}$ &2177   &217.5 & {3656} & 470.8 &  {3439} & 417.9 & {3836} & 406.6 &  {2399} & 260.1& {1530} & 185.8\bigstrut\\
    \hline
    $10^{-6}$ &  2013    & 230.9 & 1606  & 235.9 & 1589  & 221.5 & 1547  & 225.4 & 1344  & 189.6& 1082  & 168.4 \bigstrut\\
    \hline
    $10^{-8}$ &  1793    & 214.9 & 1622  & 252.7 & 1530  & 224.4 & 1562  & 234.6 & 1363  & 200.8 & 1097  & 172.7 \bigstrut\\
    \hline
    $10^{-10}$ &  1808 &  227.1  &  1599 & 260.8  & 1549  &245.3& 1565  &246.9& 1369 & 216.9 &  1100  &180.8\bigstrut\\
    \hline
    \hline
    \multicolumn{13}{|c|}{$\mu=1e-04 $} \bigstrut\\
    \hline
    \hline
    $10^{-2}$ & 961  &93.7 &  {2573} & 312.6 & 2057  & 251.5 & 1487  & 169.6 & 1367 & 137.8 & 1217  & 130.0\bigstrut\\
    \hline
    $10^{-4}$ &2146  & 217.2 & 2237  & 312.3 & 2595  & 341.6 & 2621  & 314.0 & 2044  & 242.8& 1317  & 179.6 \bigstrut\\
    \hline
    $10^{-6}$ & 2243 & 258.2 & 2102  & 307.0& 2105  & 303.9 & 1979  & 292.2 & 1810  & 272.3 & 1390  & 219.9 \bigstrut\\
    \hline
    $10^{-8}$ &2226  & 276.7 & 2057  & 329.4 & 2009  & 317.6 & 1951  & 313.6 & 1791  & 288.1 & 1444  & 250.5\bigstrut\\
    \hline
    $10^{-10}$ &   2203  & 296.2  & 2046 &361.7&  2101  &342.1& 2002 & 338.3 &1846  &307.6& 1445 & 246.7\bigstrut\\
    \hline
    \hline
    \multicolumn{13}{|c|}{$\mu=1e-05 $} \bigstrut\\
    \hline
    \hline
    $10^{-2}$ & 795   & 79.8 &  {3501} & 407.2 &  {3022} & 359.3& 1375  & 166.75 & 1156  & 122.7 & 968   & 106.7 \bigstrut\\
    \hline
    $10^{-4}$ & 1381  & 141.4& 1461  & 219.2 & 1303  & 179.2 & 1621  & 213.7 & 1198  & 153.5 & 902   & 126.6 \bigstrut\\
    \hline
    $10^{-6}$ & 2928 & 313.0 & 2343  & 352.7 & 2271  & 336.2 & 2179  & 323.8& 2001  & 297.6 & 1601  & 256.3 \bigstrut\\
    \hline
    $10^{-8}$ & 2946  & 374.2& 2401  & 380.0 & 2349  & 380.5 & 2254  & 360.8& 2099  & 345.6& 1804  & 303.9 \bigstrut\\
    \hline
    $10^{-10}$ & 2776  &   436.7 & 2503  &432.4 &  2363  &414.3 &  2350  &409.4 &  2093 & 365.7 &  1826 & 320.2\bigstrut\\
    \hline
    \end{tabular}

\end{table}

\end{document}